\documentclass[a4paper,11pt]{amsart}
\usepackage[utf8]{inputenc}
\usepackage[english]{babel}

\usepackage{amsmath}
\usepackage{amsthm}
\usepackage{amssymb}
\usepackage{amsfonts}
\usepackage{amsopn}
\usepackage{amsrefs}
\usepackage{url}
\usepackage{hyperref}
\usepackage[color=green!30]{todonotes}
\usepackage{import}

\usepackage[a4paper, top=3cm, bottom=3cm, left=2.5cm, right=2.5cm]{geometry}
\usepackage{tikz}
\usetikzlibrary{matrix, arrows}
\newtheorem{theorem}[equation]{Theorem}%[section]
\newtheorem{proposition}[equation]{Proposition}
\newtheorem{lemma}[equation]{Lemma}
\newtheorem{corollary}[equation]{Corollary}
\newtheorem{definition}[equation]{Definition}
\newtheorem{example}[equation]{Example}
\newtheorem{observation}[equation]{Observation}
\newtheorem{remark}[equation]{Remark}

\theoremstyle{remark}
\newtheorem*{ack}{Acknowledgments}
\numberwithin{equation}{section}
\def\Z{\mathbb{Z}}
\def\Q{\mathbb{Q}}
\def\R{\mathbb{R}}
\def\F{\mathbb{F}}
\def\K{\mathbb{K}}
\def\cK{\mathcal{K}}
\def\CcK{C\mathcal{K}}
\def\HcK{H\mathcal{K}}
\def\EcKP{E\cK P}
\def\EcK{E\cK}
\def\ECcK{EC\cK}

\newcommand{\fff}[1]{\mathbb{F}_{#1}}

\newcommand{\res}[2]{\operatorname{Res}^{#1}_{#2}}

\newcommand{\nonorientCross}{\raisebox{-5pt}{%% Creator: Inkscape inkscape 0.91, www.inkscape.org
\begingroup%
  \makeatletter%
  \providecommand\color[2][]{%
    \errmessage{(Inkscape) Color is used for the text in Inkscape, but the package 'color.sty' is not loaded}%
    \renewcommand\color[2][]{}%
  }%
  \providecommand\transparent[1]{%
    \errmessage{(Inkscape) Transparency is used (non-zero) for the text in Inkscape, but the package 'transparent.sty' is not loaded}%
    \renewcommand\transparent[1]{}%
  }%
  \providecommand\rotatebox[2]{#2}%
  \ifx\svgwidth\undefined%
    \setlength{\unitlength}{18.70443764bp}%
    \ifx\svgscale\undefined%
      \relax%
    \else%
      \setlength{\unitlength}{\unitlength * \real{\svgscale}}%
    \fi%
  \else%
    \setlength{\unitlength}{\svgwidth}%
  \fi%
  \global\let\svgwidth\undefined%
  \global\let\svgscale\undefined%
  \makeatother%
  \begin{picture}(1,0.91836368)%
    \put(0,0){\includegraphics[width=\unitlength,page=1]{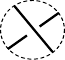}}%
  \end{picture}%
\endgroup%
}}
\newcommand{\zeroSmoothing}{\raisebox{-5pt}{%% Creator: Inkscape inkscape 0.91, www.inkscape.org
\begingroup%
  \makeatletter%
  \providecommand\color[2][]{%
    \errmessage{(Inkscape) Color is used for the text in Inkscape, but the package 'color.sty' is not loaded}%
    \renewcommand\color[2][]{}%
  }%
  \providecommand\transparent[1]{%
    \errmessage{(Inkscape) Transparency is used (non-zero) for the text in Inkscape, but the package 'transparent.sty' is not loaded}%
    \renewcommand\transparent[1]{}%
  }%
  \providecommand\rotatebox[2]{#2}%
  \ifx\svgwidth\undefined%
    \setlength{\unitlength}{18.70443764bp}%
    \ifx\svgscale\undefined%
      \relax%
    \else%
      \setlength{\unitlength}{\unitlength * \real{\svgscale}}%
    \fi%
  \else%
    \setlength{\unitlength}{\svgwidth}%
  \fi%
  \global\let\svgwidth\undefined%
  \global\let\svgscale\undefined%
  \makeatother%
  \begin{picture}(1,0.91836368)%
    \put(0,0){\includegraphics[width=\unitlength,page=1]{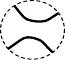}}%
  \end{picture}%
\endgroup%
}}
\newcommand{\oneSmoothing}{\raisebox{-5pt}{%% Creator: Inkscape inkscape 0.91, www.inkscape.org
\begingroup%
  \makeatletter%
  \providecommand\color[2][]{%
    \errmessage{(Inkscape) Color is used for the text in Inkscape, but the package 'color.sty' is not loaded}%
    \renewcommand\color[2][]{}%
  }%
  \providecommand\transparent[1]{%
    \errmessage{(Inkscape) Transparency is used (non-zero) for the text in Inkscape, but the package 'transparent.sty' is not loaded}%
    \renewcommand\transparent[1]{}%
  }%
  \providecommand\rotatebox[2]{#2}%
  \ifx\svgwidth\undefined%
    \setlength{\unitlength}{17.17747609bp}%
    \ifx\svgscale\undefined%
      \relax%
    \else%
      \setlength{\unitlength}{\unitlength * \real{\svgscale}}%
    \fi%
  \else%
    \setlength{\unitlength}{\svgwidth}%
  \fi%
  \global\let\svgwidth\undefined%
  \global\let\svgscale\undefined%
  \makeatother%
  \begin{picture}(1,1.08889324)%
    \put(0,0){\includegraphics[width=\unitlength,page=1]{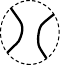}}%
  \end{picture}%
\endgroup%
}}
\newcommand{\posCross}{\raisebox{-5pt}{%% Creator: Inkscape inkscape 0.91, www.inkscape.org
\begingroup%
  \makeatletter%
  \providecommand\color[2][]{%
    \errmessage{(Inkscape) Color is used for the text in Inkscape, but the package 'color.sty' is not loaded}%
    \renewcommand\color[2][]{}%
  }%
  \providecommand\transparent[1]{%
    \errmessage{(Inkscape) Transparency is used (non-zero) for the text in Inkscape, but the package 'transparent.sty' is not loaded}%
    \renewcommand\transparent[1]{}%
  }%
  \providecommand\rotatebox[2]{#2}%
  \ifx\svgwidth\undefined%
    \setlength{\unitlength}{18.70443764bp}%
    \ifx\svgscale\undefined%
      \relax%
    \else%
      \setlength{\unitlength}{\unitlength * \real{\svgscale}}%
    \fi%
  \else%
    \setlength{\unitlength}{\svgwidth}%
  \fi%
  \global\let\svgwidth\undefined%
  \global\let\svgscale\undefined%
  \makeatother%
  \begin{picture}(1,0.91836368)%
    \put(0,0){\includegraphics[width=\unitlength,page=1]{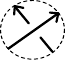}}%
  \end{picture}%
\endgroup%
}}
\newcommand{\negCross}{\raisebox{-5pt}{%% Creator: Inkscape inkscape 0.91, www.inkscape.org
\begingroup%
  \makeatletter%
  \providecommand\color[2][]{%
    \errmessage{(Inkscape) Color is used for the text in Inkscape, but the package 'color.sty' is not loaded}%
    \renewcommand\color[2][]{}%
  }%
  \providecommand\transparent[1]{%
    \errmessage{(Inkscape) Transparency is used (non-zero) for the text in Inkscape, but the package 'transparent.sty' is not loaded}%
    \renewcommand\transparent[1]{}%
  }%
  \providecommand\rotatebox[2]{#2}%
  \ifx\svgwidth\undefined%
    \setlength{\unitlength}{18.70443764bp}%
    \ifx\svgscale\undefined%
      \relax%
    \else%
      \setlength{\unitlength}{\unitlength * \real{\svgscale}}%
    \fi%
  \else%
    \setlength{\unitlength}{\svgwidth}%
  \fi%
  \global\let\svgwidth\undefined%
  \global\let\svgscale\undefined%
  \makeatother%
  \begin{picture}(1,0.91836368)%
    \put(0,0){\includegraphics[width=\unitlength,page=1]{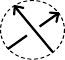}}%
  \end{picture}%
\endgroup%
}}
\newcommand{\orientRes}{\raisebox{-5pt}{%% Creator: Inkscape inkscape 0.91, www.inkscape.org
\begingroup%
  \makeatletter%
  \providecommand\color[2][]{%
    \errmessage{(Inkscape) Color is used for the text in Inkscape, but the package 'color.sty' is not loaded}%
    \renewcommand\color[2][]{}%
  }%
  \providecommand\transparent[1]{%
    \errmessage{(Inkscape) Transparency is used (non-zero) for the text in Inkscape, but the package 'transparent.sty' is not loaded}%
    \renewcommand\transparent[1]{}%
  }%
  \providecommand\rotatebox[2]{#2}%
  \ifx\svgwidth\undefined%
    \setlength{\unitlength}{18.70443764bp}%
    \ifx\svgscale\undefined%
      \relax%
    \else%
      \setlength{\unitlength}{\unitlength * \real{\svgscale}}%
    \fi%
  \else%
    \setlength{\unitlength}{\svgwidth}%
  \fi%
  \global\let\svgwidth\undefined%
  \global\let\svgscale\undefined%
  \makeatother%
  \begin{picture}(1,0.91836368)%
    \put(0,0){\includegraphics[width=\unitlength,page=1]{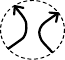}}%
  \end{picture}%
\endgroup%
}}
\newcommand{\nonorientRes}{%% Creator: Inkscape inkscape 0.91, www.inkscape.org
\begingroup%
  \makeatletter%
  \providecommand\color[2][]{%
    \errmessage{(Inkscape) Color is used for the text in Inkscape, but the package 'color.sty' is not loaded}%
    \renewcommand\color[2][]{}%
  }%
  \providecommand\transparent[1]{%
    \errmessage{(Inkscape) Transparency is used (non-zero) for the text in Inkscape, but the package 'transparent.sty' is not loaded}%
    \renewcommand\transparent[1]{}%
  }%
  \providecommand\rotatebox[2]{#2}%
  \ifx\svgwidth\undefined%
    \setlength{\unitlength}{18.70443764bp}%
    \ifx\svgscale\undefined%
      \relax%
    \else%
      \setlength{\unitlength}{\unitlength * \real{\svgscale}}%
    \fi%
  \else%
    \setlength{\unitlength}{\svgwidth}%
  \fi%
  \global\let\svgwidth\undefined%
  \global\let\svgscale\undefined%
  \makeatother%
  \begin{picture}(1,0.91836368)%
    \put(0,0){\includegraphics[width=\unitlength,page=1]{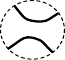}}%
  \end{picture}%
\endgroup%
}
\newcommand{\nonorientOrbit}{\raisebox{-5pt}{%% Creator: Inkscape inkscape 0.91, www.inkscape.org
\begingroup%
  \makeatletter%
  \providecommand\color[2][]{%
    \errmessage{(Inkscape) Color is used for the text in Inkscape, but the package 'color.sty' is not loaded}%
    \renewcommand\color[2][]{}%
  }%
  \providecommand\transparent[1]{%
    \errmessage{(Inkscape) Transparency is used (non-zero) for the text in Inkscape, but the package 'transparent.sty' is not loaded}%
    \renewcommand\transparent[1]{}%
  }%
  \providecommand\rotatebox[2]{#2}%
  \ifx\svgwidth\undefined%
    \setlength{\unitlength}{18.70443764bp}%
    \ifx\svgscale\undefined%
      \relax%
    \else%
      \setlength{\unitlength}{\unitlength * \real{\svgscale}}%
    \fi%
  \else%
    \setlength{\unitlength}{\svgwidth}%
  \fi%
  \global\let\svgwidth\undefined%
  \global\let\svgscale\undefined%
  \makeatother%
  \begin{picture}(1,0.91836368)%
    \put(0,0){\includegraphics[width=\unitlength,page=1]{crossing_nonoriented.pdf}}%
  \end{picture}%
\endgroup%
} \ldots \raisebox{-5pt}{%% Creator: Inkscape inkscape 0.91, www.inkscape.org
\begingroup%
  \makeatletter%
  \providecommand\color[2][]{%
    \errmessage{(Inkscape) Color is used for the text in Inkscape, but the package 'color.sty' is not loaded}%
    \renewcommand\color[2][]{}%
  }%
  \providecommand\transparent[1]{%
    \errmessage{(Inkscape) Transparency is used (non-zero) for the text in Inkscape, but the package 'transparent.sty' is not loaded}%
    \renewcommand\transparent[1]{}%
  }%
  \providecommand\rotatebox[2]{#2}%
  \ifx\svgwidth\undefined%
    \setlength{\unitlength}{18.70443764bp}%
    \ifx\svgscale\undefined%
      \relax%
    \else%
      \setlength{\unitlength}{\unitlength * \real{\svgscale}}%
    \fi%
  \else%
    \setlength{\unitlength}{\svgwidth}%
  \fi%
  \global\let\svgwidth\undefined%
  \global\let\svgscale\undefined%
  \makeatother%
  \begin{picture}(1,0.91836368)%
    \put(0,0){\includegraphics[width=\unitlength,page=1]{crossing_nonoriented.pdf}}%
  \end{picture}%
\endgroup%
}}
\newcommand{\posOrbit}{\raisebox{-5pt}{%% Creator: Inkscape inkscape 0.91, www.inkscape.org
\begingroup%
  \makeatletter%
  \providecommand\color[2][]{%
    \errmessage{(Inkscape) Color is used for the text in Inkscape, but the package 'color.sty' is not loaded}%
    \renewcommand\color[2][]{}%
  }%
  \providecommand\transparent[1]{%
    \errmessage{(Inkscape) Transparency is used (non-zero) for the text in Inkscape, but the package 'transparent.sty' is not loaded}%
    \renewcommand\transparent[1]{}%
  }%
  \providecommand\rotatebox[2]{#2}%
  \ifx\svgwidth\undefined%
    \setlength{\unitlength}{18.70443764bp}%
    \ifx\svgscale\undefined%
      \relax%
    \else%
      \setlength{\unitlength}{\unitlength * \real{\svgscale}}%
    \fi%
  \else%
    \setlength{\unitlength}{\svgwidth}%
  \fi%
  \global\let\svgwidth\undefined%
  \global\let\svgscale\undefined%
  \makeatother%
  \begin{picture}(1,0.91836368)%
    \put(0,0){\includegraphics[width=\unitlength,page=1]{crossing_pos.pdf}}%
  \end{picture}%
\endgroup%
} \ldots \raisebox{-5pt}{%% Creator: Inkscape inkscape 0.91, www.inkscape.org
\begingroup%
  \makeatletter%
  \providecommand\color[2][]{%
    \errmessage{(Inkscape) Color is used for the text in Inkscape, but the package 'color.sty' is not loaded}%
    \renewcommand\color[2][]{}%
  }%
  \providecommand\transparent[1]{%
    \errmessage{(Inkscape) Transparency is used (non-zero) for the text in Inkscape, but the package 'transparent.sty' is not loaded}%
    \renewcommand\transparent[1]{}%
  }%
  \providecommand\rotatebox[2]{#2}%
  \ifx\svgwidth\undefined%
    \setlength{\unitlength}{18.70443764bp}%
    \ifx\svgscale\undefined%
      \relax%
    \else%
      \setlength{\unitlength}{\unitlength * \real{\svgscale}}%
    \fi%
  \else%
    \setlength{\unitlength}{\svgwidth}%
  \fi%
  \global\let\svgwidth\undefined%
  \global\let\svgscale\undefined%
  \makeatother%
  \begin{picture}(1,0.91836368)%
    \put(0,0){\includegraphics[width=\unitlength,page=1]{crossing_pos.pdf}}%
  \end{picture}%
\endgroup%
}}
\newcommand{\negOrbit}{\raisebox{-5pt}{%% Creator: Inkscape inkscape 0.91, www.inkscape.org
\begingroup%
  \makeatletter%
  \providecommand\color[2][]{%
    \errmessage{(Inkscape) Color is used for the text in Inkscape, but the package 'color.sty' is not loaded}%
    \renewcommand\color[2][]{}%
  }%
  \providecommand\transparent[1]{%
    \errmessage{(Inkscape) Transparency is used (non-zero) for the text in Inkscape, but the package 'transparent.sty' is not loaded}%
    \renewcommand\transparent[1]{}%
  }%
  \providecommand\rotatebox[2]{#2}%
  \ifx\svgwidth\undefined%
    \setlength{\unitlength}{18.70443764bp}%
    \ifx\svgscale\undefined%
      \relax%
    \else%
      \setlength{\unitlength}{\unitlength * \real{\svgscale}}%
    \fi%
  \else%
    \setlength{\unitlength}{\svgwidth}%
  \fi%
  \global\let\svgwidth\undefined%
  \global\let\svgscale\undefined%
  \makeatother%
  \begin{picture}(1,0.91836368)%
    \put(0,0){\includegraphics[width=\unitlength,page=1]{crossing_neg.pdf}}%
  \end{picture}%
\endgroup%
} \ldots \raisebox{-5pt}{%% Creator: Inkscape inkscape 0.91, www.inkscape.org
\begingroup%
  \makeatletter%
  \providecommand\color[2][]{%
    \errmessage{(Inkscape) Color is used for the text in Inkscape, but the package 'color.sty' is not loaded}%
    \renewcommand\color[2][]{}%
  }%
  \providecommand\transparent[1]{%
    \errmessage{(Inkscape) Transparency is used (non-zero) for the text in Inkscape, but the package 'transparent.sty' is not loaded}%
    \renewcommand\transparent[1]{}%
  }%
  \providecommand\rotatebox[2]{#2}%
  \ifx\svgwidth\undefined%
    \setlength{\unitlength}{18.70443764bp}%
    \ifx\svgscale\undefined%
      \relax%
    \else%
      \setlength{\unitlength}{\unitlength * \real{\svgscale}}%
    \fi%
  \else%
    \setlength{\unitlength}{\svgwidth}%
  \fi%
  \global\let\svgwidth\undefined%
  \global\let\svgscale\undefined%
  \makeatother%
  \begin{picture}(1,0.91836368)%
    \put(0,0){\includegraphics[width=\unitlength,page=1]{crossing_neg.pdf}}%
  \end{picture}%
\endgroup%
}}
\newcommand{\orientResOrbit}{\raisebox{-5pt}{%% Creator: Inkscape inkscape 0.91, www.inkscape.org
\begingroup%
  \makeatletter%
  \providecommand\color[2][]{%
    \errmessage{(Inkscape) Color is used for the text in Inkscape, but the package 'color.sty' is not loaded}%
    \renewcommand\color[2][]{}%
  }%
  \providecommand\transparent[1]{%
    \errmessage{(Inkscape) Transparency is used (non-zero) for the text in Inkscape, but the package 'transparent.sty' is not loaded}%
    \renewcommand\transparent[1]{}%
  }%
  \providecommand\rotatebox[2]{#2}%
  \ifx\svgwidth\undefined%
    \setlength{\unitlength}{18.70443764bp}%
    \ifx\svgscale\undefined%
      \relax%
    \else%
      \setlength{\unitlength}{\unitlength * \real{\svgscale}}%
    \fi%
  \else%
    \setlength{\unitlength}{\svgwidth}%
  \fi%
  \global\let\svgwidth\undefined%
  \global\let\svgscale\undefined%
  \makeatother%
  \begin{picture}(1,0.91836368)%
    \put(0,0){\includegraphics[width=\unitlength,page=1]{orient_resolution.pdf}}%
  \end{picture}%
\endgroup%
} \ldots \raisebox{-5pt}{%% Creator: Inkscape inkscape 0.91, www.inkscape.org
\begingroup%
  \makeatletter%
  \providecommand\color[2][]{%
    \errmessage{(Inkscape) Color is used for the text in Inkscape, but the package 'color.sty' is not loaded}%
    \renewcommand\color[2][]{}%
  }%
  \providecommand\transparent[1]{%
    \errmessage{(Inkscape) Transparency is used (non-zero) for the text in Inkscape, but the package 'transparent.sty' is not loaded}%
    \renewcommand\transparent[1]{}%
  }%
  \providecommand\rotatebox[2]{#2}%
  \ifx\svgwidth\undefined%
    \setlength{\unitlength}{18.70443764bp}%
    \ifx\svgscale\undefined%
      \relax%
    \else%
      \setlength{\unitlength}{\unitlength * \real{\svgscale}}%
    \fi%
  \else%
    \setlength{\unitlength}{\svgwidth}%
  \fi%
  \global\let\svgwidth\undefined%
  \global\let\svgscale\undefined%
  \makeatother%
  \begin{picture}(1,0.91836368)%
    \put(0,0){\includegraphics[width=\unitlength,page=1]{orient_resolution.pdf}}%
  \end{picture}%
\endgroup%
}}
\newcommand{\nonorientResOrbit}{%% Creator: Inkscape inkscape 0.91, www.inkscape.org
\begingroup%
  \makeatletter%
  \providecommand\color[2][]{%
    \errmessage{(Inkscape) Color is used for the text in Inkscape, but the package 'color.sty' is not loaded}%
    \renewcommand\color[2][]{}%
  }%
  \providecommand\transparent[1]{%
    \errmessage{(Inkscape) Transparency is used (non-zero) for the text in Inkscape, but the package 'transparent.sty' is not loaded}%
    \renewcommand\transparent[1]{}%
  }%
  \providecommand\rotatebox[2]{#2}%
  \ifx\svgwidth\undefined%
    \setlength{\unitlength}{18.70443764bp}%
    \ifx\svgscale\undefined%
      \relax%
    \else%
      \setlength{\unitlength}{\unitlength * \real{\svgscale}}%
    \fi%
  \else%
    \setlength{\unitlength}{\svgwidth}%
  \fi%
  \global\let\svgwidth\undefined%
  \global\let\svgscale\undefined%
  \makeatother%
  \begin{picture}(1,0.91836368)%
    \put(0,0){\includegraphics[width=\unitlength,page=1]{nonorient_resolution.pdf}}%
  \end{picture}%
\endgroup%
 \ldots %% Creator: Inkscape inkscape 0.91, www.inkscape.org
\begingroup%
  \makeatletter%
  \providecommand\color[2][]{%
    \errmessage{(Inkscape) Color is used for the text in Inkscape, but the package 'color.sty' is not loaded}%
    \renewcommand\color[2][]{}%
  }%
  \providecommand\transparent[1]{%
    \errmessage{(Inkscape) Transparency is used (non-zero) for the text in Inkscape, but the package 'transparent.sty' is not loaded}%
    \renewcommand\transparent[1]{}%
  }%
  \providecommand\rotatebox[2]{#2}%
  \ifx\svgwidth\undefined%
    \setlength{\unitlength}{18.70443764bp}%
    \ifx\svgscale\undefined%
      \relax%
    \else%
      \setlength{\unitlength}{\unitlength * \real{\svgscale}}%
    \fi%
  \else%
    \setlength{\unitlength}{\svgwidth}%
  \fi%
  \global\let\svgwidth\undefined%
  \global\let\svgscale\undefined%
  \makeatother%
  \begin{picture}(1,0.91836368)%
    \put(0,0){\includegraphics[width=\unitlength,page=1]{nonorient_resolution.pdf}}%
  \end{picture}%
\endgroup%
}
\newcommand{\zeroSmoothingOrbit}{\raisebox{-5pt}{%% Creator: Inkscape inkscape 0.91, www.inkscape.org
\begingroup%
  \makeatletter%
  \providecommand\color[2][]{%
    \errmessage{(Inkscape) Color is used for the text in Inkscape, but the package 'color.sty' is not loaded}%
    \renewcommand\color[2][]{}%
  }%
  \providecommand\transparent[1]{%
    \errmessage{(Inkscape) Transparency is used (non-zero) for the text in Inkscape, but the package 'transparent.sty' is not loaded}%
    \renewcommand\transparent[1]{}%
  }%
  \providecommand\rotatebox[2]{#2}%
  \ifx\svgwidth\undefined%
    \setlength{\unitlength}{18.70443764bp}%
    \ifx\svgscale\undefined%
      \relax%
    \else%
      \setlength{\unitlength}{\unitlength * \real{\svgscale}}%
    \fi%
  \else%
    \setlength{\unitlength}{\svgwidth}%
  \fi%
  \global\let\svgwidth\undefined%
  \global\let\svgscale\undefined%
  \makeatother%
  \begin{picture}(1,0.91836368)%
    \put(0,0){\includegraphics[width=\unitlength,page=1]{0-smoothing.pdf}}%
  \end{picture}%
\endgroup%
} \ldots \raisebox{-5pt}{%% Creator: Inkscape inkscape 0.91, www.inkscape.org
\begingroup%
  \makeatletter%
  \providecommand\color[2][]{%
    \errmessage{(Inkscape) Color is used for the text in Inkscape, but the package 'color.sty' is not loaded}%
    \renewcommand\color[2][]{}%
  }%
  \providecommand\transparent[1]{%
    \errmessage{(Inkscape) Transparency is used (non-zero) for the text in Inkscape, but the package 'transparent.sty' is not loaded}%
    \renewcommand\transparent[1]{}%
  }%
  \providecommand\rotatebox[2]{#2}%
  \ifx\svgwidth\undefined%
    \setlength{\unitlength}{18.70443764bp}%
    \ifx\svgscale\undefined%
      \relax%
    \else%
      \setlength{\unitlength}{\unitlength * \real{\svgscale}}%
    \fi%
  \else%
    \setlength{\unitlength}{\svgwidth}%
  \fi%
  \global\let\svgwidth\undefined%
  \global\let\svgscale\undefined%
  \makeatother%
  \begin{picture}(1,0.91836368)%
    \put(0,0){\includegraphics[width=\unitlength,page=1]{0-smoothing.pdf}}%
  \end{picture}%
\endgroup%
}}
\newcommand{\oneSmoothingOrbit}{\raisebox{-5pt}{%% Creator: Inkscape inkscape 0.91, www.inkscape.org
\begingroup%
  \makeatletter%
  \providecommand\color[2][]{%
    \errmessage{(Inkscape) Color is used for the text in Inkscape, but the package 'color.sty' is not loaded}%
    \renewcommand\color[2][]{}%
  }%
  \providecommand\transparent[1]{%
    \errmessage{(Inkscape) Transparency is used (non-zero) for the text in Inkscape, but the package 'transparent.sty' is not loaded}%
    \renewcommand\transparent[1]{}%
  }%
  \providecommand\rotatebox[2]{#2}%
  \ifx\svgwidth\undefined%
    \setlength{\unitlength}{17.17747609bp}%
    \ifx\svgscale\undefined%
      \relax%
    \else%
      \setlength{\unitlength}{\unitlength * \real{\svgscale}}%
    \fi%
  \else%
    \setlength{\unitlength}{\svgwidth}%
  \fi%
  \global\let\svgwidth\undefined%
  \global\let\svgscale\undefined%
  \makeatother%
  \begin{picture}(1,1.08889324)%
    \put(0,0){\includegraphics[width=\unitlength,page=1]{1-smoothing.pdf}}%
  \end{picture}%
\endgroup%
} \ldots \raisebox{-5pt}{%% Creator: Inkscape inkscape 0.91, www.inkscape.org
\begingroup%
  \makeatletter%
  \providecommand\color[2][]{%
    \errmessage{(Inkscape) Color is used for the text in Inkscape, but the package 'color.sty' is not loaded}%
    \renewcommand\color[2][]{}%
  }%
  \providecommand\transparent[1]{%
    \errmessage{(Inkscape) Transparency is used (non-zero) for the text in Inkscape, but the package 'transparent.sty' is not loaded}%
    \renewcommand\transparent[1]{}%
  }%
  \providecommand\rotatebox[2]{#2}%
  \ifx\svgwidth\undefined%
    \setlength{\unitlength}{17.17747609bp}%
    \ifx\svgscale\undefined%
      \relax%
    \else%
      \setlength{\unitlength}{\unitlength * \real{\svgscale}}%
    \fi%
  \else%
    \setlength{\unitlength}{\svgwidth}%
  \fi%
  \global\let\svgwidth\undefined%
  \global\let\svgscale\undefined%
  \makeatother%
  \begin{picture}(1,1.08889324)%
    \put(0,0){\includegraphics[width=\unitlength,page=1]{1-smoothing.pdf}}%
  \end{picture}%
\endgroup%
}}

\newcommand{\ckh}{\operatorname{CKh}}
\newcommand{\eckh}{\operatorname{ECKh}}

\newcommand{\kh}{\operatorname{Kh}}
\newcommand{\ekh}{\operatorname{EKh}}
\newcommand{\ekhp}{\operatorname{EKhP}}
\newcommand{\khp}{\operatorname{KhP}}
\newcommand{\jones}{\operatorname{J}}
\newcommand{\lee}{\operatorname{Lee}}
\newcommand{\elee}{\operatorname{ELee}}
\def\Lee{\lee}
\def\ELee{\elee}
\def\CKh{\ckh}
\def\Kh{\kh}
\def\EKh{\ekh}
\def\EKhP{\ekhp}
\def\KhP{\khp}

\def\ECKh{\eckh}
\DeclareMathOperator{\Ext}{Ext}
\DeclareMathOperator{\Hom}{Hom}

\DeclareMathOperator{\im}{Im}
\DeclareMathOperator{\rk}{rk}
\DeclareMathOperator{\LeeP}{LeeP}
\DeclareMathOperator{\ELeeP}{ELeeP}
\DeclareMathOperator{\BN}{BN}
\DeclareMathOperator{\EBN}{EBN}

\DeclareMathOperator{\EBNP}{EBNP}

\DeclareMathOperator{\chr}{char}
\DeclareMathOperator{\Or}{Or}
\DeclareMathOperator{\Gr}{Gr}
\DeclareMathOperator{\EJ}{EJ}
\DeclareMathOperator{\DJones}{DJ}
\newcommand{\bnbracket}[1]{[\kern-1.5pt [ #1 ]\kern-1.5pt]_{\operatorname{BN}}}
\newcommand{\khbracket}[1]{[\kern-1.5pt [ #1 ]\kern-1.5pt]_{\operatorname{Kh}}}

\newcommand{\One}{\operatorname{1 \kern-3.75pt 1}}
\newcommand{\ol}[1]{\overline{#1}}
\newcommand{\wt}[1]{\widetilde{#1}}
\title{Khovanov homology and periodic links}
\author{Maciej Borodzik}
\address{Institute of Mathematics, Polish Academy of Science, ul. \'Sniadeckich 8,
00-656 Warsaw, Poland}
\address{Institute of Mathematics, University of Warsaw, ul. Banacha 2,
02-097 Warsaw, Poland}
\email{mcboro@mimuw.edu.pl}
\author{Wojciech Politarczyk}
\address{Institute of Mathematics, University of Warsaw, ul. Banacha 2,
02-097 Warsaw, Poland}
\email{wpolitarczyk@mimuw.edu.pl}
\date{\today}

\date{\today}
\makeatletter
\@namedef{subjclassname@2010}{\textup{2010} Mathematics Subject Classification}
\makeatother
 %, secondary: }

\begin{document}

\begin{abstract}
Based on the results of the second author, we define an equivariant version of Lee and Bar-Natan homology for periodic links and show that
there exists an equivariant spectral sequence from the equivariant Khovanov homology to equivariant Lee homology. As a result we obtain
new obstructions for a link to be periodic. These obstructions generalize previous results
of Przytycki and of the second author.
\end{abstract}
\maketitle

\section{Overview}

\subsection{Statement of results}
Let $L\subset\R^3$ be a link. Let $m\in\Z$, $m>1$. 
We say that $L$ is \emph{$m$-periodic} if $L$ is preserved by a rotational symmetry of $\R^{3}$ of order $m$ and $L$ is disjoint
from the fixed point set of the symmetry. A diagram $D$ of an $m$-periodic link $L$ is called $m$-periodic if $D$ is a diagram of $L$
and the action of $\Z_m$ on $\R^3$ that preserves $L$, descends to the rotation of $\R^2$ preserving $D$. We require $D$ to be disjoint from the center of the rotation of $\R^2$. Any $m$-periodic link admits an $m$-periodic diagram.

The question which links are periodic has attracted a lot of interest of knot theorists. One of the first results is due to Murasugi \cite{Murasugi},
who gave a criterion in terms of Alexander polynomials. The criterion, generalized for links by Sakuma \cite{Sakuma} is still one of the most effective criteria
for periodicity of links. Next cornerstone was a result of Sakuma \cite{Sakuma2}, who essentially solved the 
periodicity problem for composite knots and links. The main
focus is thus 
on periodicity of prime knots and links. Finally Weeks \cite{Weeks} found an algorithm for finding a canonical triangulation of a hyperbolic three-manifold
and used it, among other applications, to detect symmetries of hyperbolic three-manifolds, in particular, to detect symmetries of hyperbolic links. The criterion,
implemented in a computer package SnapPy \cite{SnapPy} is extremely effective, even if it is used in a conservative way (that is, we let SnapPy calculate the symmetry
group of a complement of the link $L$, and if this group does not have an element of order $m$, we conclude that $L$ cannot be $m$ periodic).

Even though SnapPy can obstruct periodicity of almost all knots, it cannot help with periodicity of non-hyperbolic knots and links. While 
the non-hyperbolic knots form only a very small portion of all knots (for example, if one looks at the amount of the non-hyperbolic knots among
the prime knots with bounded number of crossings) and the only prime non-hyperbolic
knots for which the periodicity question is not completely answered are cable knots, non-hyperbolic links occur much more often. Moreover, it is often
useful to have periodicity obstructions coming from different link invariants, because often this gives insight into the structure of the links invariants
for periodic links.

In fact, a lot of invariants have been used to obstruct periodicity of a given link:
Jones polynomial \cite{Traczyk}, HOMFLYPT polynomials \cite{Przytycki-periodic,Traczyk2}, homology of branched covers \cite{Naik},
twisted Alexander polynomials \cite{HillmanLivingstonNaik} and many others. The most recent approach
of Jabuka and Naik uses Heegaard--Floer homologies \cite{JabukaNaik}, there is also an obstruction for strong invertibility coming from Khovanov
homology of tangles given by Watson \cite{Watson}. 
Our paper follows the idea that knot homologies should give an effective tool for obstructing periodicity of knots.

In \cite{Politarczyk-Khovanov} 
the second author constructed an equivariant version of Khovanov homology. The equivariant Khovanov homology was used in \cite{Politarczyk-Jones}
to generalize the periodicity
obstruction of Traczyk \cite{Traczyk}. The two criteria in \cite{Politarczyk-Jones} and \cite{Przytycki-periodic} are stated in terms of Jones polynomials. 
In the present article we generalize these results to obtain a criterion that involves the Khovanov polynomial of a link. The main advantage is
that the Khovanov polynomial by the definition has non-negative coefficients, a feature that is lost when one passes to the Jones polynomial.
However, we must stress that the new criterion also shares a weakness of the periodicity criterion involving
the Jones polynomial: it cannot be used to obstruct periods $2$ and $3$; see Section~\ref{sec:period3}. We refer also to a recent preprint
of Zhang \cite{Zhang} for the use of Khovanov homology to obstruct 2--periodicity.

Throughout the paper $p$ and $r$ denote prime numbers. We work either over the field $\F=\Q$ or $\F=\F_{r}$ a finite field of order $r$. 
The number $s(K,\F)$ is the $s$-invariant of $K$ derived from the Lee or Bar-Natan theory, see \cite{Lee,BarNatan}.
We use the convention that the width of $\Kh(K;\F)$ is the maximum of $i-2j-(i'-2j')$ such that $\Kh^{i,j}(K;\F)$
and $\Kh^{i',j'}(K;\F)$ are non-empty.

The following result is the main theorem of the present paper. As the statement is rather technical, in this place we state it for knots. The general case of links is
given in Section~\ref{sec:links}. We also refer the reader to Section~\ref{sec:knotexample} for elaborating an example, which might be more enlightening than
the statement of Theorem~\ref{thm:periodicity_criterion} itself.

\begin{theorem}\label{thm:periodicity_criterion}
  Let $K$ be a $p^n$-periodic knot, where $p$ is an odd prime. Suppose that $\F = \Q$ or $\F_{r}$, for a prime $r$ such that $r \neq p$, and  $r$ has maximal order in the multiplicative group mod $p^{n}$. Set $c=1$ if $\F=\F_2$ and $c=2$ otherwise. Then 
the Khovanov polynomial $\KhP(K;\F)$ decomposes as
\begin{equation}\label{eq:presentation1}
\KhP(K;\F)=\mathcal{P}_0+\sum_{j=1}^n (p^j-p^{j-1})\mathcal{P}_j,
\end{equation}
where 
  \[\mathcal{P}_{0}, \mathcal{P}_{1}, \ldots, \mathcal{P}_{n} \in \Z[q^{\pm 1}, t^{\pm 1}],\]
  are Laurent polynomials such that
\begin{enumerate}
 \item $\mathcal{P}_{0} = q^{s(K,\F)}(q+q^{-1}) + \sum_{j=1}^\infty (1+tq^{2cj})\mathcal{S}_{0j}(t,q)$, and the polynomials $\mathcal{S}_{0j}$ have non-negative coefficients; \label{item3:main_thm_1}
 \item $\mathcal{P}_{k} = \sum_{j=1}^{\infty}(1+tq^{2cj})\mathcal{S}_{kj}(t,q)$ and the polynomials $\mathcal{S}_{kj}$ have non-negative coefficients for $1 \leq k \leq n$, \label{item4:main_thm_1}
 \item $\mathcal{P}_{k}(-1,q) - \mathcal{P}_{k+1}(-1,q) \equiv \mathcal{P}_{k}(-1,q^{-1}) - \mathcal{P}_{k+1}(-1,q^{-1})\pmod{q^{p^{n-k}} - q^{-p^{n-k}}}$; \label{item5:main_thm_1}
\item If the width of $\Kh(K;\F)$ is equal to $w$, then $S_{kj}=0$ for $j>\frac{c}{2}w$.\label{item6:main_thm_1}
\end{enumerate}
\end{theorem}

Notice that for given Khovanov polynomial, checking whether there exists a presentation \eqref{eq:presentation1} 
 satisfying conditions (1)--(4) of Theorem~\ref{thm:periodicity_criterion} 
 involves a finite number of cases to check. This is because of positivity of Khovanov
polynomials and of polynomials $\mathcal{P}_k$.

The structure of the paper is the following.
Section~\ref{sec:equiv} defines equivariant Khovanov homology theories and defines the Lee and Bar-Natan spectral sequences in an equivariant setting.
At the end of this section, the equivariant Jones polynomial are reviewed. The proof of Theorem~\ref{thm:periodicity_criterion} is given in 
Section~\ref{sec:bigproof}. Then we pass to applications. In Section~\ref{sec:compare} we review other known periodicity criteria for knots.
In Section~\ref{sec:nonalternating} we show how effective is the new criterion in comparizon with various criteria for periodicity. In Section~\ref{sec:snappy}
we show that, for knots and not for links, Theorem~\ref{thm:periodicity_criterion} is beaten by SnapPy in almost all cases. Finally in Section~\ref{sec:period3}
we show a limitation of Theorem~\ref{thm:periodicity_criterion} for periods 2 and 3.

In Section~\ref{sec:links} we review the periodicity problem for links. First we review existing criteria for links, in a few cases it turns out that a criterion
works for links even though it is stated in an article having 'periodicity of knots' in the title. In Section~\ref{sec:period_links} we prove an analogue
of Theorem~\ref{thm:periodicity_criterion} for links.  In Section~\ref{sec:exlinks} we show that the new periodicity criterion is effective in some cases, when
the other criteria fail, including SnapPy.

 In Section~\ref{sec:period3} we explain
why the criterion does not work for periods 2 and 3. Then
we pass to concrete applications of Theorem~\ref{thm:periodicity_criterion}. The first application is the knot $15n135221$ for which we show how to apply
Theorem~\ref{thm:periodicity_criterion} in an efficient way. Although we do not prove this rigorously, the example indicates that there is an algorithm checking
whether given Khovanov polynomial satisfies the statement of Theorem~\ref{thm:periodicity_criterion} in time $O(n^{p/2})$, where $p$ is the period and $n$
is the total rank of Khovanov homology. Section~\ref{sec:compare} reviews several known periodicity obstruction and then Section~\ref{sec:nonalternating}
compares the criterion of Theorem~\ref{thm:periodicity_criterion} with other known
criteria, taking as a testing class non-alternating prime knots with 12 to 15 crossings and checking for knots with period 5. 
In particular, among others, the knot $15n135221$ is shown to pass Murasugi's criterion for the Alexander polynomial, Przytycki--Traczyk criterion for the HOMFLYPT
polynomial and the Naik's homological criterion, but fails to Theorem~\ref{thm:periodicity_criterion}.
Some basic facts
from representation theory used throughout the paper are gathered in Appendix~\ref{appendix}.

\begin{ack}  
The authors have profited a lot from conversations and e-mails with many people. We are especially grateful to Stefan Friedl,
Slaven Jabuka, Chuck Livingston, Allison Miller, Swatee Naik and Jozef Przytycki. We would like to thank Alexander Shumakovitch for sharing with us a file with Khovanov polynomials for knots with up to 15 crossings and to Paul Kirk for sharing with us a program to compute twisted Alexander polynomials. 
The KnotKit~\cite{KnotKit} was an invaluable help in checking examples.

The first author was supported by the National Science Center grant 2016/22/E/ST1/00040.
The second author was supported by the National Science Center grant 2016/20/S/ST1/00369.
\end{ack}

\subsection{Example. Knot $15n135221$}\label{sec:knotexample}

Theorem~\ref{thm:periodicity_criterion} can be applied directly if the Khovanov polynomial is known and
it always involves checking a finite number of cases. However if the Khovanov polynomial has large coefficients (like $10$) and has a large number of
coefficients, running through all the possibilities of presenting $\KhP$ as a sum of polynomials $(1+tq^{4j})S_{kj}(t,q)$ might be rather cumbersome. The number possibilities grows exponentially with the total rank of the Khovanov homology. 
Below we present an algorithm that significantly reduces the number of possibilities. It was implemented in \cite{KnotKit-WP}.

\begin{figure}
\definecolor{linkcolor0}{rgb}{0.15, 0.45, 0.45}
\begin{tikzpicture}[line width=2, line cap=round, line join=round, scale=0.7]
  \begin{scope}[color=linkcolor0]
    \draw (6.38, 8.61) .. controls (6.88, 8.61) and (7.40, 8.43) .. 
          (7.40, 8.00) .. controls (7.40, 7.50) and (6.76, 7.40) .. (6.18, 7.40);
    \draw (6.18, 7.40) .. controls (5.85, 7.40) and (5.51, 7.40) .. (5.17, 7.40);
    \draw (4.78, 7.40) .. controls (3.70, 7.40) and (2.55, 7.11) .. 
          (2.55, 6.18) .. controls (2.55, 5.52) and (3.10, 4.97) .. (3.76, 4.97);
    \draw (3.76, 4.97) .. controls (4.10, 4.97) and (4.44, 4.97) .. (4.78, 4.97);
    \draw (5.17, 4.97) .. controls (5.77, 4.97) and (6.18, 5.55) .. (6.18, 6.18);
    \draw (6.18, 6.18) .. controls (6.18, 6.52) and (6.18, 6.86) .. (6.18, 7.20);
    \draw (6.18, 7.59) .. controls (6.18, 7.93) and (6.18, 8.27) .. (6.18, 8.61);
    \draw (6.18, 8.61) .. controls (6.18, 9.41) and (7.09, 9.82) .. 
          (8.00, 9.82) .. controls (9.71, 9.82) and (9.82, 7.57) .. 
          (9.82, 5.58) .. controls (9.82, 3.68) and (9.82, 1.34) .. (8.61, 1.34);
    \draw (8.61, 1.34) .. controls (7.06, 1.34) and (5.51, 1.34) .. (3.96, 1.34);
    \draw (3.57, 1.34) .. controls (2.30, 1.34) and (1.34, 2.46) .. (1.34, 3.76);
    \draw (1.34, 3.76) .. controls (1.34, 5.31) and (1.34, 6.86) .. (1.34, 8.41);
    \draw (1.34, 8.80) .. controls (1.34, 9.57) and (2.28, 9.82) .. 
          (3.16, 9.82) .. controls (4.07, 9.82) and (4.97, 9.41) .. (4.97, 8.61);
    \draw (4.97, 8.61) .. controls (4.97, 8.20) and (4.97, 7.80) .. (4.97, 7.40);
    \draw (4.97, 7.40) .. controls (4.97, 7.06) and (4.97, 6.72) .. (4.97, 6.38);
    \draw (4.97, 5.99) .. controls (4.97, 5.65) and (4.97, 5.31) .. (4.97, 4.97);
    \draw (4.97, 4.97) .. controls (4.97, 4.57) and (4.97, 4.17) .. (4.97, 3.76);
    \draw (4.97, 3.76) .. controls (4.97, 3.10) and (5.52, 2.55) .. 
          (6.18, 2.55) .. controls (6.82, 2.55) and (7.40, 2.97) .. (7.40, 3.57);
    \draw (7.40, 3.96) .. controls (7.40, 5.01) and (7.29, 6.18) .. (6.38, 6.18);
    \draw (5.99, 6.18) .. controls (5.65, 6.18) and (5.31, 6.18) .. (4.97, 6.18);
    \draw (4.97, 6.18) .. controls (4.34, 6.18) and (3.76, 5.77) .. (3.76, 5.17);
    \draw (3.76, 4.78) .. controls (3.76, 4.50) and (3.76, 4.23) .. (3.76, 3.96);
    \draw (3.76, 3.57) .. controls (3.76, 2.82) and (3.76, 2.08) .. (3.76, 1.34);
    \draw (3.76, 1.34) .. controls (3.76, 0.35) and (5.03, 0.13) .. 
          (6.18, 0.13) .. controls (7.32, 0.13) and (8.61, 0.17) .. (8.61, 1.14);
    \draw (8.61, 1.54) .. controls (8.61, 2.62) and (8.32, 3.76) .. (7.40, 3.76);
    \draw (7.40, 3.76) .. controls (6.65, 3.76) and (5.91, 3.76) .. (5.17, 3.76);
    \draw (4.78, 3.76) .. controls (4.44, 3.76) and (4.10, 3.76) .. (3.76, 3.76);
    \draw (3.76, 3.76) .. controls (3.02, 3.76) and (2.28, 3.76) .. (1.54, 3.76);
    \draw (1.14, 3.76) .. controls (0.17, 3.76) and (0.13, 5.05) .. 
          (0.13, 6.18) .. controls (0.13, 7.34) and (0.35, 8.61) .. (1.34, 8.61);
    \draw (1.34, 8.61) .. controls (2.49, 8.61) and (3.63, 8.61) .. (4.78, 8.61);
    \draw (5.17, 8.61) .. controls (5.44, 8.61) and (5.72, 8.61) .. (5.99, 8.61);
  \end{scope}
\end{tikzpicture}
\caption{Knot $15n135221$.}\label{fig:15n135221}
\end{figure}
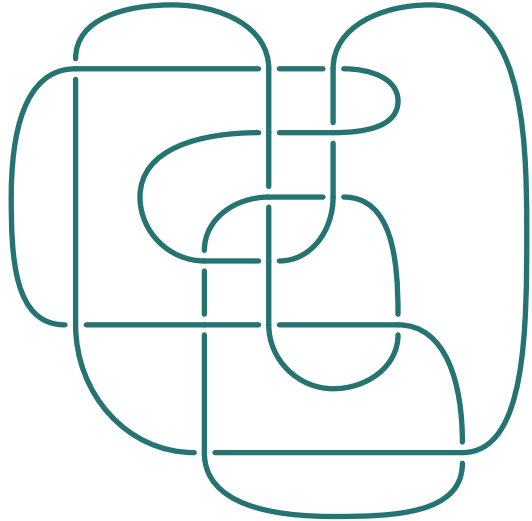
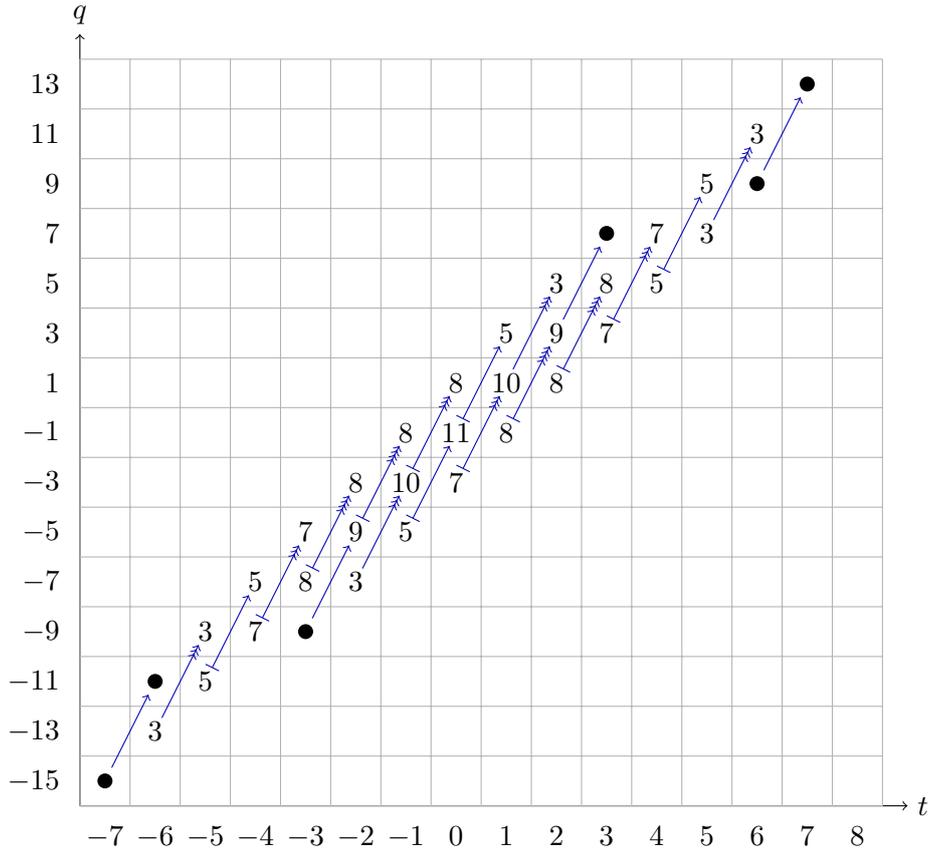
\begin{figure}
\begin{tikzpicture}[scale=.66]
  \draw[->] (0,0) -- (16.5,0) node[right] {$t$};
  \draw[->] (0,0) -- (0,15.5) node[above] {$q$};
  \draw[step=1,color=black!30] (0,0) grid (16,15);
  \draw (0.5,-.2) node[below] {$-7$};
  \draw (1.5,-.2) node[below] {$-6$};
  \draw (2.5,-.2) node[below] {$-5$};
  \draw (3.5,-.2) node[below] {$-4$};
  \draw (4.5,-.2) node[below] {$-3$};
  \draw (5.5,-.2) node[below] {$-2$};
  \draw (6.5,-.2) node[below] {$-1$};
  \draw (7.5,-.2) node[below] {$0$};
  \draw (8.5,-.2) node[below] {$1$};
  \draw (9.5,-.2) node[below] {$2$};
  \draw (10.5,-.2) node[below] {$3$};
  \draw (11.5,-.2) node[below] {$4$};
  \draw (12.5,-.2) node[below] {$5$};
  \draw (13.5,-.2) node[below] {$6$};
  \draw (14.5,-.2) node[below] {$7$};
  \draw (15.5,-.2) node[below] {$8$};
  \draw (-.2,0.5) node[left] {$-15$};
  \draw (-.2,1.5) node[left] {$-13$};
  \draw (-.2,2.5) node[left] {$-11$};
  \draw (-.2,3.5) node[left] {$-9$};
  \draw (-.2,4.5) node[left] {$-7$};
  \draw (-.2,5.5) node[left] {$-5$};
  \draw (-.2,6.5) node[left] {$-3$};
  \draw (-.2,7.5) node[left] {$-1$};
  \draw (-.2,8.5) node[left] {$1$};
  \draw (-.2,9.5) node[left] {$3$};
  \draw (-.2,10.5) node[left] {$5$};
  \draw (-.2,11.5) node[left] {$7$};
  \draw (-.2,12.5) node[left] {$9$};
  \draw (-.2,13.5) node[left] {$11$};
  \draw (-.2,14.5) node[left] {$13$};
  \fill (0.5, 0.5) circle (.15);
  \draw (1.5, 1.5) node {$3$};
  \fill (1.5, 2.5) circle (.15);
  \draw (2.5, 2.5) node {$5$};
  \draw (2.5, 3.5) node {$3$};
  \draw (3.5, 3.5) node {$7$};
  \draw (3.5, 4.5) node {$5$};
  \fill (4.5, 3.5) circle (.15);
  \draw (4.5, 4.5) node {$8$};
  \draw (4.5, 5.5) node {$7$};
  \draw (5.5, 4.5) node {$3$};
  \draw (5.5, 5.5) node {$9$};
  \draw (5.5, 6.5) node {$8$};
  \draw (6.5, 5.5) node {$5$};
  \draw (6.5, 6.5) node {$10$};
  \draw (6.5, 7.5) node {$8$};
  \draw (7.5, 6.5) node {$7$};
  \draw (7.5, 7.5) node {$11$};
  \draw (7.5, 8.5) node {$8$};
  \draw (8.5, 7.5) node {$8$};
  \draw (8.5, 8.5) node {$10$};
  \draw (8.5, 9.5) node {$5$};
  \draw (9.5, 8.5) node {$8$};
  \draw (9.5, 9.5) node {$9$};
  \draw (9.5, 10.5) node {$3$};
  \draw (10.5, 9.5) node {$7$};
  \draw (10.5, 10.5) node {$8$};
  \fill (10.5, 11.5) circle (.15);
  \draw (11.5, 10.5) node {$5$};
  \draw (11.5, 11.5) node {$7$};
  \draw (12.5, 11.5) node {$3$};
  \draw (12.5, 12.5) node {$5$};
  \fill (13.5, 12.5) circle (.15);
  \draw (13.5, 13.5) node {$3$};
  \fill (14.5, 14.5) circle (.15);
  \draw[color=blue!70!black,->] (0.634, 0.768) -- (1.366, 2.232);
  \draw[color=blue!70!black,->>>] (1.634, 1.768) --  (2.366, 3.232);
  \draw[color=blue!70!black,|->] (2.634, 2.768) --  (3.366, 4.232);
  \draw[color=blue!70!black,|->>>] (3.634, 3.768) --  (4.366, 5.232);
  \draw[color=blue!70!black,->] (4.634, 3.768) -- (5.366, 5.232);
  \draw[color=blue!70!black,|->>>>] (4.634, 4.768) --  (5.366, 6.232);
  \draw[color=blue!70!black,->>>] (5.634, 4.768) --  (6.366, 6.232);
  \draw[color=blue!70!black,|->>>>] (5.634, 5.768) --  (6.366, 7.232);
  \draw[color=blue!70!black,|->] (6.634, 5.768) --  (7.366, 7.232);
  \draw[color=blue!70!black,|->>>] (6.634, 6.768) --  (7.366, 8.232);
  \draw[color=blue!70!black,|->>>] (7.634, 6.768) --  (8.366, 8.232);
  \draw[color=blue!70!black,|->] (7.634, 7.768) --  (8.366, 9.232);
  \draw[color=blue!70!black,|->>>>] (8.634, 7.768) --  (9.366, 9.232);
  \draw[color=blue!70!black,->>>] (8.634, 8.768) --  (9.366, 10.232);
  \draw[color=blue!70!black,|->>>>] (9.634, 8.768) --  (10.366, 10.232);
  \draw[color=blue!70!black,->] (9.634, 9.768) -- (10.366, 11.232);
  \draw[color=blue!70!black,|->>>] (10.634, 9.768) -- (11.366, 11.232);
  \draw[color=blue!70!black,|->] (11.634, 10.768) --  (12.366, 12.232);
  \draw[color=blue!70!black,->>>] (12.634, 11.768) --  (13.366, 13.232);
  \draw[color=blue!70!black,->] (13.634, 12.768) -- (14.366, 14.232);
\end{tikzpicture}
\caption{The Khovanov homology of $15n135221$ over $\F_3$. The arrows denote differentials on the first page of the Lee spectral sequence. 
The number of tips indicates the rank of the coimage. For example
\raisebox{0.1cm}{\protect\tikz\protect\draw[->>] (0,0.0) -- (0.4,0.0);}
indicate the rank is $2$. The numbers $5$ or more are indicated with a bar at the beginning, so 
\raisebox{0.1cm}{\protect\tikz\protect\draw[|->] (0,0) -- (0.4,0);}
indicates rank $5$,
\raisebox{0.1cm}{\protect\tikz\protect\draw[|->>] (0,0) -- (0.4,0); }
indicates rank $6$ and so on.}\label{fig:15n135221_1}
\end{figure}
Consider the knot $15n135221$ depicted in Figure~\ref{fig:15n135221}. This knot has Alexander polynomial $1$ and passes the HOMFLYPT criterion for period $5$.
The Khovanov homology over $\F_3$ and the 
differential on the $E_1$ page of the Lee spectral sequence are depicted in Figure~\ref{fig:15n135221_1}.
The Lee spectral sequence degenerates at $E_2$ and the Rasmussen $s$-invariant is $0$.
This information allows us to write the Khovanov polynomial of $15n135221$ in the following form
\begin{align*}
  \KhP &= q + q^{-1} + (1 + tq^4)(t^{-7}q^{-15} + 3t^{-6}q^{-13} + t^{-5}q^{-11} + 3t^{-4}q^{-9} + t^{-3}q^{-9} + 3t^{-2}q^{-7}\\
       &+ t^{-1}q^{-5} + 3t^{-1}q^{-3} + q^{-3} + q^{-1} + 3tq + t^2q^3 + 3 t^3q^3 + t^4q^5 + 3t^5q^7 + t^6q^9 \\
       &+ 4 (t^{-5} q^{-11} + t^{-4} q^{-9} + 2t^{-3} q^{-7} + 2t^{-2} q^{-5} + t^{-1}q^{-5} + t^{-1} q^{-3} + 2t q^{-1} \\
       &+ q^{-3} + q^{-1} + 2 t^2q + t^3 q^3 + t^4 q^5)).
\end{align*}
The above decomposition gives us candidates for $\mathcal{S}_{01}$ and $\mathcal{S}_{11}$. Namely
\begin{align*}
  \mathcal{S}_{01}' &= t^{-7}q^{-15} + 3t^{-6}q^{-13} + t^{-5}q^{-11} + 3t^{-4}q^{-9} + t^{-3}q^{-9} + 3t^{-2}q^{-7}\\
                     &+ t^{-1}q^{-5} + 3t^{-1}q^{-3} + q^{-3} + q^{-1} + 3tq + t^2q^3 + 3 t^3q^3 + t^4q^5 + 3t^5q^7 + t^6q^9, \\
  \mathcal{S}_{11}' &= t^{-5} q^{-11} + t^{-4} q^{-9} + 2t^{-3} q^{-7} + 2t^{-2} q^{-5} + t^{-1}q^{-5} + t^{-1} q^{-3} + 2t q^{-1} \\
       &+ q^{-3} + q^{-1} + 2 t^2q + t^3 q^3 + t^4 q^5.
\end{align*}
According to Theorem~\ref{thm:periodicity_criterion}(3) we define
$\wt{\Xi}(q)=(q+q^{-1}+(1+tq^4)(\mathcal{S}_{01}'-\mathcal{S}_{11}')|_{t=-1}$ and set
$\Xi:=(\wt{\Xi}(q)-\wt{\Xi}(q^{-1}))\bmod{q^5-q^{-5}}$.
We have then
\begin{equation}
  \label{eq:difference-reduces}
  \Xi = -10 q + 5 q^3 - 5 q^7 + 10 q^9.
\end{equation}
Notice that $\Xi$ has all its coefficients divisible by $5$. This is correct, because otherwise the knot $15n135221$ would not pass the HOMFLYPT criterion.
Since $\Xi\neq 0$, the chosen candidates $\mathcal{S}_{01}'$ and $\mathcal{S}_{11}'$ for $\mathcal{S}_{01}$ and  $\mathcal{S}_{11}$ do not satisfy the conditions from Theorem~\ref{thm:periodicity_criterion} which means that we have to try to modify $\mathcal{S}_{01}'$ and $\mathcal{S}_{11}'$ in such a way that the resulting candidates do. 
Possible changes, that is, changes that preserve properties (1) and (2) of Theorem~\ref{thm:periodicity_criterion}, are 
\begin{equation}\label{eq:possible_changes}
\begin{split}
\mathcal{S}_{11}'&\mapsto \mathcal{S}_{11}'-\delta\\
\mathcal{S}_{01}'&\mapsto \mathcal{S}_{01}'+4\delta,
\end{split}
\end{equation}
where $\delta$ is a polynomial with non-negative coefficients such that $\mathcal{S}_{11}'-\delta$ has non-negative coefficients. The number
of possible choices of $\delta$ is finite, even if it is quite large: each term $t^iq^j$ entering $\mathcal{S}_{11}'$ with coefficient $c_{i,j}>0$ gives
a factor $(c_{i,j}+1)$ and the number of possible choices of $\delta$ is the product of all such factors. In this case we get $3^4\cdot 2^8=20736$
possibilities.

In order to reduce the number of possibilities we use the following argument. We calculate that if $\delta=a t^iq^j$ then after the change \eqref{eq:possible_changes} we have $\Xi\mapsto\Xi+aT_{ij}$, where 
\[T_{ij}=(-1)^i5(-q^{-j-4}+q^{-j}-q^{j}+q^{j+4})\bmod (q^5-q^{-5}).\]
Given that $T_{ij}$ is defined only modulo $q^5-q^{-5}$, the remainder of $T_{ij}$ is equal to $(-1)^iR_{j'}$, where $j'=j\bmod 10$. Moreover, since we work with knots, $j$ and $j'$ are always odd. Therefore,

\begin{align*}
  R_1 = R_5 &= 5 (q - q^9), \\
  R_3 &= 10 (q^3 - q^7), \\
  R_7 = R_9 &= 5 (-q - q^3 + q^7 + q^9).
\end{align*}

For each polynomial $\delta$ with non-negative coefficients, such that $S_{11}'-\delta$ has non-negative coefficients, $\Xi$ changes (as described in \eqref{eq:possible_changes}) by a linear combination of $R_1, R_3, R_7$. More precisely, different choices of $\delta$ change $\Xi$ by $-a_1 R_1 - a_3 R_3 - a_7 R_7$, where
the condtion that $\delta$ and $S_{11}'-\delta$ have non-negative coefficients translates into the following restrictions on $a_1$, $a_3$ and $a_7$.
$$a_1 \in \{-1, 0, 1, 2, 3, 4, 5, 6\}, \quad a_3 \in \{-3, -2, -1, 0\}, \quad a_7 \in \{-4, -3, -2, -1, 0, 1, 2\}.$$
This reduces the number of cases to check to $224$ if one uses brute force method. However, it can also be easily verified by hand that $\Xi$ cannot be written as a linear combination $a_1 R_1 + a_3 R_3 + a_7 R_7$ with the restrictions as above.

\begin{remark}
For general prime period $p$ the number of possible $R_i$ is equal to $\frac{p+1}{2}$. Each of the $a_i$ is definitely bounded by the total rank of Khovanov homology, hence the number of cases to check grows at most like $n^{p/2}$, where $n$ is the rank of the Khovanov polynomial. 
\end{remark}

\section{Equivariant Knovanov theories}\label{sec:equiv}

\subsection{Equivariant Khovanov-like homology}
\label{sec:equiv-khovanov-like-homology}
From now on we start gradually developing theory that will eventually lead to the proof of Theorem~\ref{thm:periodicity_criterion}.

We will recall the construction of the equivariant Khovanov homology. Our aim is to define also equivariant Bar-Natan and 
equivariant Lee homology. In order to conduct this construction in a uniform way, we will use the language of Frobenius systems
introduced in \cite{Khovanov}.

\begin{definition}
A \emph{rank two Frobenius system} $\mathcal{F} = (R, A, \epsilon, \Delta)$ consists of:
\begin{itemize}
\item a commutative ring $R$,
\item an $R$-algebra $A$ which is also a free $R$-module of rank $2$,
\item an $A$-bimodule homomorphism $\Delta \colon A \to A \otimes_{R} A$ that is cocommutative and coassociative,
\item an $R$-linear map $\epsilon \colon A \to R$ which is a counit for $\Delta$.
\end{itemize}
\end{definition}
Khovanov in \cite{Khovanov} constructed a universal rank two Frobenius system $\mathcal{F}$ with
\[R = \Z[h,t], \quad A = R[X]/(X^2 - hX - t), \quad \deg(h) = -2, \quad \deg(t) = -4,\]
and comultiplication and counit 
\begin{align*}
  &\epsilon(1) = 0, \quad \Delta(1) = 1 \otimes X + X \otimes 1 + h 1 \otimes 1, \\
  &\epsilon(X) = 1, \quad \Delta(X) = X \otimes X + t 1 \otimes 1.
\end{align*}
Universality means that any rank two Frobenius system $\mathcal{F}' = (R', A', \epsilon', \Delta')$, that can be used to produce a functorial (up to sign) link invariant, can be obtained from $\mathcal{F}$ by a base change, 
that is, for any such rank two Frobenius system $\mathcal{F}'=(R',A',\epsilon',\Delta')$ there exists a unital ring homomorphism
\[\psi \colon R \to R',\]
such that $A' = A \otimes_{R} R'$ and $\epsilon'$ and $\Delta'$ are induced counit and comultiplication.

One can use a Frobenius system to define a Khovanov-like
link homology. The underlying chain complex is constructed via a cube of resolution and the differential is a map induced by multiplication,
respectively comultiplication, in the algebra $A$. 
Khovanov \cite{Khovanov} proves that a Frobenius system $\mathcal{F}$ yields a homology theory which is an invariant of links and is functorial (up to sign) with respect to link cobordisms. Moreover, any Frobenius system obtained from $\mathcal{F}$ by a base change also yields a link invariant functorial (up to sign) with respect to link cobordisms.

\begin{example}\label{example:frob_system-khovanov-homology}
  Ordinary Khovanov homology (with coefficients in a field $\F$) can be obtained from a Frobenius system determined by the base change
  \[\psi \colon R \to \F, \quad \psi(h) = \psi(t) = 0.\]
  Since $\ker \psi$ is a homogeneous ideal the resulting homology theory is graded.
\end{example}

\begin{example}\label{example:frob_system-lee-homology}
  Lee homology \cite{Lee}, with coefficients in $\F = \Q$ or $\F_{p}$ for an odd prime $p$, can be obtained from the following Frobenius system
  \[\psi \colon R \to \F, \quad \psi(h) = 0, \quad \psi(t) = 1.\]
\end{example}

\begin{example}\label{example:frob_system-Bar-Natan-theory}
  Filtered Bar-Natan theory \cite{BarNatan} can be obtained from Frobenius system
  \[\psi \colon R \to \F_{2}, \quad \psi(h) = 1, \quad \psi(t) = 0.\]
\end{example}

\begin{definition}
  An assignment $D \mapsto \CcK^{\ast}(D;\F)$, which maps a link diagram $D$ to a chain complex over a field is called a Khovanov-like theory, if $\CcK^{\ast}(D;\F)$ is obtained from a Frobenius system $\mathcal{F}'$ with $R' = \F$ a field.
\end{definition}

Every Khovanov-like theory satisfies the following conditions.
\begin{itemize}
\item The underlying space of $\CcK^{\ast}(D;\F)$ is the same as of the Khovanov chain complex $\CKh^{\ast}(D;\F)$;
\item The differential $d_{\cK}$ preserves the homological grading of $\CKh^{\ast}(D;\F)$, and does not decrease the quantum grading. In
particular $d_{\cK}$ is a filtered map;
\item Every Reidemeister move that changes $D$ into $D'$ induces a filtered chain homotopy equivalence $\CcK^{\ast}(D;\F) \to \CcK^{\ast}(D';\F)$;
\end{itemize}

Suppose $D$ is an $m$-periodic diagram of an $m$-periodic link $L$
and $\CcK^{\ast}(D;\F)$ is a Khovanov-like theory. Let $\Lambda_m=\F[\Z_m]$.
\begin{theorem}\label{thm:invariance}
The action of $\Z_m$ on $D$ turns $\CcK^{\ast}(D;\F)$ into a chain complex over $\Lambda_m$. The groups
\[\EcK(L;M)=\Ext_{\Lambda_m}(M;\CcK^{\ast}(D;\F)).\]
are invariants of the equivariant isotopy class of a periodic link.
\end{theorem}
\begin{proof}
The proof is essentially a repetition of proof of \cite[Theorem 4.3]{Politarczyk-Khovanov}. The fact that $\CcK^{\ast}(D;\K)$ has
a structure of a chain complex over $\Lambda_m$ follows from a simple observation that $\Z_m$ acts on the cube of resolution and commutes with the differential.

Next we use \cite[Theorem 3.14]{Politarczyk-Khovanov} to show that if $D$ and $D'$ differ by an
equivariant Reidemeister move, then there exists a chain homotopy equivalences $\CcK^{\ast}(D;\F) \to \CcK^{\ast}(D';\F)$ of complexes of $\Lambda_m$
modules. Finally we apply a result from homological algebra, as in \cite[Proposition 2.15]{Politarczyk-Khovanov} to see that $\EcK(L;M)$ is invariant.
\end{proof}

In this way, taking as a starting point original Khovanov theory, Lee theory,  or the Bar-Natan theory, see Examples~\ref{example:frob_system-khovanov-homology}, \ref{example:frob_system-lee-homology} and \ref{example:frob_system-Bar-Natan-theory} above, 
we can define the equivariant Khovanov homology $\EKh(L;M)$, the equivariant Lee homology $\ELee(L;M)$, the equivariant
Bar-Natan homology $\EBN(L;M)$, respectively.

It is convenient to present $\EcK(L;M)$ as a homology of a certain chain complex. To this end, 
choose a projective resolution of $M$ over $\Lambda_m$:
\begin{center}
\begin{tikzpicture}
\matrix (m) [matrix of math nodes, column sep=3em, row sep=10em, minimum width=2em,nodes={minimum height=1em, anchor=center}]
{
\ldots & P_3 & P_2 & P_1 & M \\};
\path [-stealth]
(m-1-1) edge node [above] {$\delta_3$} (m-1-2)
(m-1-2) edge node [above] {$\delta_2$} (m-1-3)
(m-1-3) edge node [above] {$\delta_1$} (m-1-4)
(m-1-4) edge node [above] {$\delta_0$} (m-1-5);
\end{tikzpicture}
\end{center}

Then $\EcK(L;M)$ is the homology of the following complex
\begin{equation}\label{eq:chaincomplex}
\ECcK^{i}(L;M)=\bigoplus_{a+b=i} \Hom(P_a,\CcK^{b}(D)),
\end{equation}
with the differential $d_{\cK}+\delta$. The chain complex on the right hand side of \eqref{eq:chaincomplex} is a filtered complex with respect to the
quantum grading. Moreover, if $d_\cK$ actually preserves the quantum grading, then $\ECcK$ is a graded chain complex as well.

From the construction of $\ECcK$ we obtain the following result.
\begin{proposition}\label{prop:cartan}
There exists a Cartan-Eilenberg spectral sequence, whose $E_2$ page is given by
\begin{equation}\label{eq:cartankhovanov}
E_2^{i,k}=\Ext_{\Lambda_{m}}^i(M,\HcK^{k}(L;\F)),
\end{equation}
where $\HcK$ is the homology of $\CcK$,
and which converges to $\EcK^{i+k}(L;M)$.

Furthermore, if $m$ is invertible in $\F$, then the spectral sequence degenerates at the $E_2$ page and we have $\EcK^i(L;M)=\Hom_{\Lambda_m}(M,\HcK^{i}(D;\F))$.
\end{proposition}
\begin{proof}
The first part follows from elementary homological algebra, see \cite{Weibel}. The second part follows from Corollary~\ref{cor:extvanish}.
\end{proof}
\begin{corollary}\label{cor:trivial_on_scalars}
If $m$ is invertible in $\F$, then $\EcK^i(L;\Lambda_m)=\HcK^i(L;\F)$.
\end{corollary}

The quantum filtration of the chain complex $\ECcK^i(L;M)$ descends to a filtration of the homology $\EcK$. For an integer $k$ let $\Gr^k\EcK^i$ be the graded part.
Define the following polynomial
\begin{equation}\label{eq:ekhp}
\EcKP(L;M)=\sum_{i,j}t^iq^k\dim_\F \Gr^k\EcK^i(L;M).
\end{equation}
This polynomial is called the equivariant polynomial of the Khovanov-like theory $\cK$. In particular we recover the equivariant Khovanov polynomial $\EKhP(L;M)$, the
equivariant Lee polynomial $\ELeeP(L;M)$ and the equivariant Bar--Natan polynomial $\EBNP(L;M)$.

The equivariant Khovanov polynomial  $\EcKP$ has always non-negative coefficients. 
In some cases, that is, for some concrete modules $M$, we can obtain further restrictions on the coefficients of $\EcKP$. In the following
result we use the notation of Appendix~\ref{appendix}.

\begin{proposition}\label{prop:divisib}
Suppose the period $m$ of the link $L$ is invertible in $\F$.
\begin{itemize}
\item[(a)] If $\F=\Q$ and $M=\Q(\xi_d)$ for some $d \mid m$, then all the coefficients of $\EcKP(L;M)$ are divisible by $\varphi(d)=\dim_{\Q}\Q(\xi_d)$.
\item[(b)] If $\F=\F_r$ and $M=V_\chi$ for some fixed $\chi\in C(m,r)$, then all the coefficients of $\EcKP(L;M)$ are divisible
by $\dim_{\F} V_\chi$.
\end{itemize}
\end{proposition}
\begin{proof}
By Corollary~\ref{cor:extvanish}
the Cartan--Eilenberg spectral sequence \eqref{eq:cartankhovanov} degenerates at the $E_2$ page and we obtain
\[\EcK^{i}(L;M)=\Hom_{\Lambda_m}(M;\HcK^{i}(L;\F)).\]
Notice that in both cases (a) and (b) $M$ is a field and $\EcK^{i}(L;M)$ is an $M$-vector space. Likewise, the quotients of the filtration 
are also $M$-vector spaces. Therefore
\[\dim_{\F} \Gr^{k}\EcK^{i}(L;M) = \dim_{\F} M \cdot \dim_{M} \Gr^{k}\EcK^{i}(L;M)\]
which finishes the proof.
\end{proof}

\subsection{Properties of the equivariant Lee and Bar-Natan homology}\label{sec:lee}
Fix an orientation $\mathcal{O}_0$ of a link $L$. Denote by $\Or(L)$ the set of all orientations of $L$ and for $\mathcal{O} \in \Or(L)$ let $lk(\mathcal{O})$ denote the total linking number of the underlying oriented link. Define, for $i \in \Z$,
$$\Or_{i}(L) = \{\mathcal{O} \in \Or(L) \colon lk(\mathcal{O}) - lk(\mathcal{O}_0) = i/2\}.$$

The classical (non-equivariant) Lee and Bar-Natan homology of a link can be explicitly calculated in the following way.
\begin{theorem}[\cite{Lee,Rasmussen,Turner}]\label{thm:lee-bar-natan-homology}
  Let $L$ be a link. Suppose that $L$ consists of $k$ components, then $\dim_{\F}\Lee^{\ast}(L;\F) = \dim_{\fff{2}}\BN^{\ast}(L) = 2^{k}$. Moreover, there exists a canonical basis
  $$\{x_{\mathcal{O}} \colon \mathcal{O} \in \Or(L)\}$$
  of $\Lee^{\ast}(L,\F)$ and $\BN^{\ast}(L,\F_2)$ so that
  \begin{align*}
    \Lee^i(L,\F) &= span_{\F}\{x_{\mathcal{O}} \colon \mathcal{O} \in \Or_{i}(L)\},\\
    \BN^i(L,\F_2) &= span_{\F_2}\{x_{\mathcal{O}} \colon \mathcal{O} \in \Or_{i}(L)\}.
  \end{align*}
\end{theorem}
Filtration on Lee and Bar-Natan homology allows one to define the $s$-invariant of a link, which we denote by $s(L;\F)$ and define as an average of filtration gradings of canonical generators corresponding to $\mathcal{O}_{0}$ and its reverse $\overline{\mathcal{O}}_{0}$, i.e. we invert orientation of every component of $L$.

We pass now to the equivariant case. We keep assuming that $\F$ is either $\Q$ or $\F_{r}$ for some prime number $r$.  
Assume that the fixed orientation $\mathcal{O}_0$ is invariant under the action of $\Z_m$. For any $d \mid m$ denote by $\Or^{d}(L)$ the set of orientations with isotropy group isomorphic to $\Z_d$, i.e.
$$\Or^d(L) = \{\mathcal{O} \in \Or(L) \colon \{g \in \Z_m \colon g \cdot \mathcal{O} = \mathcal{O}\} = \Z_d\}.$$
Define also $\Or_i^d(L) = \Or_{i}(L) \cap \Or^d(L)$.

Theorem~\ref{thm:lee-bar-natan-homology} specifies to the following result in the equivariant case.
\begin{proposition}\label{prop:equiv-lee-homology-computation-semisimple-case}
  If $L$ is an $m$-periodic link and $d \mid m$, then
  \begin{align*}
    \ELee^{i}(L;\Q(\xi_{\frac{m}{d}})) &\cong \bigoplus_{d' \mid d} \bigoplus_{\mathfrak{O} \in \Or_i^{d’}(L)/\Z_m} \Q(\xi_{\frac{m}{d}}), \\
    \ELee^{i}(L;V_{\chi}) &\cong \bigoplus_{d' \mid d} \bigoplus_{\mathfrak{O} \in \Or_i^{d’}(L)/\Z_m} V_{\chi}, \quad \chi \in C(m,r)_{d},\, r \neq 2 \\
    \EBN^{i}(L;V_{\chi}) &\cong \bigoplus_{d' \mid d} \bigoplus_{\mathfrak{O} \in \Or_i^{d’}(L)/\Z_m} V_{\chi}, \quad \chi \in C(m,2)_{d}.
  \end{align*}
\end{proposition}
\begin{proof}
  Choose an orbit $\mathfrak{O} \in \Or(L)_i^{d'}/\Z_m$, for some $d' \mid m$. Notice that $X_{\mathfrak{O}} = span_{\F}\{x_{\mathcal{O}} \colon \mathcal{O} \in \mathfrak{O}\}$ contributes to $\Lee^{i}(L;\F)$ a summand isomorphic to $\F[\Z_{m}/\Z_{d'}]$. Since
  \[\F[\Z_{m}/\Z_{d'}] = \bigoplus_{d \mid \frac{m}{d'}} \F(\xi_{d}),\]
  the contribution of $X_{\mathfrak{O}}$ to $\ELee^i(L;\F(\xi_{k}))$ (see Appendix~\ref{appendix} for the notation) is non-trivial provided that $k \mid \frac{m}{d'}$.
  Moreover, $X_{\mathfrak{O}}$ contributes $\Hom_{\Lambda_{m}}(\F(\xi_{d}), \F(\xi_{d}))$ to $\ELee^i(L;\F(\xi_{d}))$, which, by~\cite{CurtisReiner}*{Lemma 3.19}, is isomorphic to $\F(\xi_{d})$.
  If $\F = \Q$ we are done, otherwise use Proposition~\ref{prop:representations-finite-characteristic} to finish the proof.
\end{proof}

\begin{remark}
  It is worth to notice that canonical bases of summands in the decomposition of equivariant Lee and Bar-Natan homology from the previous theorem are given by certain linear combinations of vectors $\{x_{\mathcal{O}} \colon \mathcal{O} \in \Or(L)\}.$ For instance if $\{\mathcal{O}_1,\ldots,\mathcal{O}_{m/d}\} \subset \Or_i^d(L)$ is an orbit then the respective summand in $\ELee^i(L, \F)$ will be generated by $(x_{\mathcal{O}_1} + \ldots + x_{\mathcal{O}_{m/d}})$. In general bases of equivariant Lee and Bar-Natan homology are determined by characters of the respective representations see~\cite{CurtisReiner}*{Proposition 9.17}. 
\end{remark}

\begin{corollary}\label{cor:equiv_lee_homology_knot}
  If $L$ is an $m$-periodic link, with the property that the action of components of $L$, then
  \begin{align*}
    \ELee(L,\F) &= \Lee(L,\F), \\
    \EBN(L,\F_2) &= \BN(L,\F_2)
  \end{align*}
  and zero otherwise.
  In particular if $L = K$ is a knot we obtain
  \begin{align*}
    \ELee(K;\F) &= \ELee^{0}(K;\F) = \F^{2}, \\
    \EBN(K;\F_{2}) &= \EBN^{0}(K;\F_{2}) = \F_{2}^{2},
  \end{align*}
  and zero otherwise.
\end{corollary}
\begin{proof}
  This is an easy consequence of Proposition~\ref{prop:equiv-lee-homology-computation-semisimple-case}, because $\Or_d(L) = \emptyset$, for $d>1$.
\end{proof}

\subsection{Equivariant Lee and Bar-Natan spectral sequence}\label{sec:leess}

Comparing Frobenius system from Example \ref{example:frob_system-khovanov-homology} with Frobenius systems from Examples \ref{example:frob_system-lee-homology} and \ref{example:frob_system-Bar-Natan-theory} we see that the non-homogeneous parts of Lee and Bar-Natan differentials yield endomorphisms
\begin{align*}
  \Phi_{Lee} &\colon \Kh(L;\F) \to \Kh(L;\F), \\
  \Phi_{BN}  &\colon \Kh(L;\F_{2}) \to \Kh(L;\F_{2})
\end{align*}
of bidegree $(1,4)$ and $(1,2)$, respectively.

\begin{remark}
The difference in grading of $\Phi_{Lee}$ and $\Phi_{BN}$ 
is the reason why Theorem~\ref{thm:periodicity_criterion} has two variants: one for fields of characteristic different than $2$,
the other for fields of characteristic equal to $2$.
\end{remark}

Presenting the differential $d_{Lee}$ in Lee homology as $\partial+\Phi_{Lee}$, where $\partial$ is the Khovanov differential and presenting
the differential $d_{BN}$ in Bar-Natan homology as $\partial+\Phi_{BN}$ yields the following well-known
spectral sequences
$$\{E_{u}^{Lee}(\F), d_{u}^{Lee}\}, \quad \{E_{u}^{BN}(\F), d_{u}^{BN}\}$$
converging to Lee homology $\Lee(L;\F)$ and Bar-Natan homology $\BN(L;\F)$, respectively, with $E_{1}$ pages
isomorphic to $\Kh(L;\F)$, see \cite{Rasmussen,Turner}. Analogous spectral sequence can be constructed for equivariant Lee and Bar-Natan homology.
The construction is rather straightforward, however in the future we will need to control the bidegree of differentials and hence we give a precise statement and a proof.

\begin{proposition}\label{prop:equivariantspectral}
  Let $L$ be an $m$-periodic link. Let $\F$ be either $\Q$ or $\F_{r}$ for a prime number $r$. Set as usual $\Lambda_m=\F[\Z_m]$.
  \begin{enumerate}
  \item If the characteristic of $\F$ is different than $2$ and $M$ is an $\Lambda_m$-module, then there exists a spectral sequence
    $$\{E_{u}^{\ELee}(M), d_{u}^{\ELee}\}$$
    converging to $\ELee(L;M)$ such that the $E_{1}$ page is isomorphic to $\EKh(L;M)$. 
  Moreover the differential $d_{u}^{\ELee}$ is of bidegree $(1,4u)$.
  \item If the characteristic of $\F$ is $2$ and $M$ is an $\Lambda_m$-module, then there exists a spectral sequence
    $$\{E_{u}^{\EBN}(M), d_{u}^{\EBN}\}$$
    converging to $\EBN(L;M)$ with $E_{1}$ page isomorphic to $\EKh(L;M)$. The differential $d_{u}^{\EBN}$ is of bidegree $(1,2u)$.
  \end{enumerate}
\end{proposition}
\begin{proof}
  We will only give the proof of the first assertion. The proof of the second assertion is analogous.
  Let $\ECKh(L;M)$ be as in \eqref{eq:chaincomplex}. As we described above, 
homology of this complex with respect to the differential $\partial+\delta+\Phi_{Lee}$ gives $\ELee(L;M)$. 
The associated graded complex has differential $\partial+\delta$ and calculates $\EKh(L;M)$. Therefore there exists an
 associated spectral sequence, which abutes  in $\ELee(L;M)$, whose $E_{1}$ page equal to $\EKh(L;M)$.

The differential $d_1^{Lee}$ is just the map induced by $\Phi_{Lee}$ on $\EKh(L;M)$ and has grading $(1,4)$ by the definition of $\Phi_{Lee}$.
Gradings of higher differentials are calculated in a standard way.
\end{proof}

Recall that $\ELeeP$, $\EBNP$ and $\EKhP$ were defined in \eqref{eq:ekhp} as polynomials encoding the equivariant theories.
The spectral sequences described above will allow us to relate the polynomials $\ELeeP$ and $\EBNP$ with the
equivariant Khovanov polynomial $\EKhP$.

To achieve this we will use a more general result.
\begin{proposition}[\cite{McCleary}]\label{prop:gradedpolynomial}
Suppose $(C^{*,*},\partial)$ is a graded complex with homology $H^{*,*}$. Assume that 
$\Phi\colon C^{*}\to C^{*}$ increases the first grading by $1$ and the second
grading by $c>0$. Suppose $d+\Phi$ is also a differential, let $H_{\Phi}^*$ denote the homology of the filtered complex $(C^{*,*},d+\Phi)$.
Denote by $E^{i,j}_u$ the spectral sequence, whose $E_1$ page is $H^{i,j}$ and which converges to $H_{\Phi}^{i+j}$. Denote by
$P_u$ the polynomial
\[P_u=\sum_{i,j\in\Z^2}\rk E^{i,j}_u t^iq^j.\]
Then there exist polynomials $R_1,R_2,\ldots$ with integer non-negative coefficients such that $P_u=P_{u+1}+(1+tq^{cu})R_u$.
\end{proposition}
\begin{proof}
For $u=1,\ldots$ let $d_u$ be the differential on the $u$-th page of the spectral sequence. It changes the first grading by $1$
and the second grading by $cu)$.
Consider the following part of the $u$-th page

\begin{tikzpicture}
\fill[color=white] (-8,0.5) rectangle (6,1); % a dirty trick to center the picture
\matrix (s)[matrix of math nodes, column sep=3em,minimum width=4em,nodes={minimum height=1em, anchor=center}]
{
\ldots &  E^{i,j}_{u} & E^{i+1,j+cu}_{u} & \ldots\\};
\path[-stealth]
(s-1-1) edge node [above] {$d_{u}$} (s-1-2)
(s-1-2) edge node [above] {$d_{u}$} (s-1-3)
(s-1-3) edge node [above] {$d_{u}$} (s-1-4);
\end{tikzpicture}

By additivity of ranks we have
\begin{multline*}
\rk E^{i,j}_{u+1}=\rk\ker d_{u}(E^{i,j}_{u})-\rk \im d_{u}(E^{i-1,j-cu}_{u})=\\
=\rk E^{i,j}_{u}-\rk\im d_{u}(E^{i,j}_{u})-\rk\im d_{u}(E^{i-1,j-cu}_{u}).
\end{multline*}
We set now
\begin{equation}\label{eq:rk}
R_{u}=\sum_{i,j\in\Z}\rk\im d_{u+1}(E^{i,j}_{u})t^iq^j.
\end{equation}
The statement follows immediately.
\end{proof}
As a corollary we obtain the following result.

%%%%%%%%%%%%%%%%%%%%%%%%%%%%%%%
\begin{proposition}\label{prop:ekhpandeleep}
Let $L$ be an $m$-periodic link. For any field $\F$ and a $\Lambda_{m}$-module $M$ there exist polynomials $R_1, R_{2},\ldots$ with non-negative integer coefficients such that 
\begin{align*}
  \EKhP(L;M) &= \ELeeP(L;M) + (1+tq^4)R_1 + (1+tq^8) R_2 + \ldots, & \textrm{ if }\chr(\F) \neq 2\\
  \EKhP(L;M) &= \EBNP(L;M) + (1+tq^2)R_1 + (1+tq^4) R_2 + \ldots, & \textrm{ if }\chr(\F) = 2.
\end{align*}
Moreover, if $m$ is invertible in $\F$, $\chr(\F)\neq 2$ and the width of $L$ is $w$, then then $R_{w/2+1}=\ldots=0$. If $m$ is
invertible in $\F$ and $\chr(\F)=2$, then $R_{w+1}=\ldots=0$.
\end{proposition}
\begin{remark}
The polynomials $R_1,\ldots$ obviously depend on the choice of the field $\F$, but we do not indicate this dependence to simplify the notation.
\end{remark}
\begin{proof}[Proof of Proposition~\ref{prop:ekhpandeleep}]
Suppose $\chr(\F)\neq 2$.
By Proposition~\ref{prop:equivariantspectral} we infer that there exists a spectral sequence, whose $E_2$ page is $\EKh(L;M)$
and $E_\infty$ page is $\Gr\ELee(L;M)$. By Proposition~\ref{prop:gradedpolynomial} there exist polynomials
$R_2,\ldots,$ such that
\[\EKhP(L;M)=\ELeeP(L;M)+(1+tq^4)R_1+(1+t^2q^8)R_2+\ldots.\]
The sum on the right hand side is finite, because the spectral sequence collapses at a finite stage.

Suppose that $m$ is invertible in $\F$ and the width of $L$ is $w$. By Corollary~\ref{cor:extvanish} we infer that 
$\EcK^i(L;M)=\Hom_{\Lambda_m}(M;\HcK^{i}(L;\F))$, hence the width of $\EcK^i(L;M)$ is at most $w$. In particular,
the width of $E_u^{\ELee}$ is at most $w$ as well, because the width cannot increase in the spectral sequence.
By Proposition~\ref{prop:equivariantspectral}, $d_u^{\ELee}$ changes grading by $(1,4u)$. If the width is $w$, then $d_u$ must be zero
for $u>w/2$.

In case $\chr(\F)=2$ the proof is analogous. Only the last part requires some extra attention. The differential $d_u$ changes
the grading by $(1,2u)$, so $d_{w+1}=d_{w+2}=\ldots=0$, consequently $R_{w+1}=\ldots=0$.
\end{proof}

\subsection{Difference Jones polynomials}\label{sec:diffjones}
In Section~\ref{sec:equiv-khovanov-like-homology} we defined, for every equivariant Khovanov-like theory, the associated polynomial $\EcKP$. In this section we focus our attention on $t = -1$ specialization of $\EKhP$ that recovers the graded Euler characteristic of the associated equivariant Khovanov-like theory. 
The case $\F=\Q$ is discussed in more detail in \cite{Politarczyk-Jones}. We recall quickly this case and then indicate differences between the case $\F=\Q$
and $\F=\F_r$.

The Khovanov polynomial specifies to the Jones polynomial via the formula
\[\KhP(L)|_{t=-1}=\jones(L).\]
By analogy, we define the equivariant Jones polynomial via
\[\EJ(L;M):=\EKhP(L;M)|_{t=-1}.\]
If $M=\Q(\xi_d)$, Proposition \ref{prop:divisib} implies that $\EJ(L;M)$ is divisible by $\dim_{\Q}M$, therefore in the remainder part of the paper we will use the following notation
\[\EJ_{d}(L) = \frac{1}{\varphi(d)}\EJ(L;\Q(\xi_{d})),\]
for $d \mid m$ and $\varphi(d)$ is the Euler's totient function; see \eqref{eq:Eulers-totient-function}.

Next proposition provides a fundamental relation between the equivariant Jones polynomials and the classical Jones polynomial. It is a consequence of
Corollary~\ref{cor:trivial_on_scalars}.
\begin{proposition}[\cite{Politarczyk-Jones}]\label{prop:fundamental-property-equiv-jones-pol}
For an $m$-periodic link $L$ we have
\[\jones(L)=\sum_{d|m}\varphi(d)\EJ_d(L).\]
\end{proposition}

Suppose now that $m=p^n$, where $p$ is a prime. For any integer $0\le k\le n$  define the \emph{difference Jones polynomial} 
\[\DJones_{k}(L) =
  \begin{cases}
    \EJ_{p^k}(L)-\EJ_{p^{k+1}}(L), & k < n, \\
    \EJ_{p^{n}}(L),              & k = n.
  \end{cases}
  \]
Difference polynomials satisfy the following skein relation.
\begin{theorem}[\cite{Politarczyk-Jones}]\label{thm:skein relation}
  Let $L$ be a $p^{n}$-periodic link.
\begin{enumerate}
\item $\DJones_{0}$ satisfies the following version of the skein relation
\begin{align*}
 &q^{-2p^n} \DJones_{0}\left(\posOrbit\right) - q^{2p^n} 
 \DJones_{0}\left(\negOrbit\right) = \\
 &= \left(q^{-p^n} - q^{p^n}\right) \DJones_{0}\left(\orientResOrbit\right),
\end{align*}
where $\posOrbit$, $\negOrbit$ and $\orientResOrbit$ denote an orbit of positive, negative and 
orientation preserving resolutions of crossing, respectively.
\item for any $0 \leq s \leq n$, $\DJones_{s}$ satisfies the following congruences
\begin{align*}
 &q^{-2p^{n}}\DJones_{n-s}\left(\posOrbit\right) - 
 q^{2p^{n}}\DJones_{n-s}\left(\negOrbit\right) \equiv \\
 &\equiv \left(q^{-p^n} - q^{p^n}\right) \DJones_{n-s}\left(\orientResOrbit\right) 
 \pmod{q^{p^{s}} - q^{-p^{s}}}.
\end{align*}
\end{enumerate}
\end{theorem}

% A periodic link can always be transformed into its mirror using equivariant crossing changes. Using the skein relation from Theorem \ref{thm:skein relation} we obtain the following result.
% \footnote{[MB] Do we need this result at all?}
% \begin{proposition}[\cite{Politarczyk-Jones}]\label{prop:periodicity-crit-char-zero}
% Let $L$ be a $p^n$-periodic link and let $\ol{L}$ be its mirror. Suppose that for any $i,j \in \Z$ we have $\dim_{\Q}\Kh(L;\Q) < p^{k}-p^{k-1}$, for some $1 \leq k \leq n$, then
% \[\jones(L)\equiv \jones(\ol{L})\bmod I_k,\]
% where $I_k$ is the ideal generated by
% $$q^{p^{n}} - q^{-p^{n}}, p\left(q^{p^{n-1}} - q^{-p^{n-1}}\right), \ldots, p^{k-1} \left(q^{p^{n-k+1}} - q^{-p^{n-k+1}}\right).$$
% \end{proposition}

The above construction, conducted first in \cite{Politarczyk-Jones}, can be generalized to the case when $\F$ is a field of characteristics $r \neq 0$.
Suppose $L$ is an $m$--periodic link and assume
that $m$ is invertible in $\F = \F_{r}$.
Below we will use notation from Appendix~\ref{appendix}.
Denote
\[\EJ_{\chi,r}(L) = \frac{1}{f_{\frac{m}{d}, r}}\EJ(L;V_{\chi}),\]
for $\chi \in C(m,r)_{d}$ and $d \mid m$. By Proposition \ref{prop:divisib} the polynomial $\EJ_{\chi,r}$ has integer coefficients.
Moreover, for any $d \mid m$ we set
\[\EJ_{\frac{m}{d},r}(L) = \EJ(L;\F_{r}[\xi_{\frac{m}{d}}]) = \sum_{\chi \in C(m,r)_{d}} \EJ(L;V_{\chi}).\]
An analogue of Proposition~\ref{prop:fundamental-property-equiv-jones-pol} in positive characteristic is given in the next proposition.
\begin{proposition}\label{prop:fund-property-equiv-jones-pol-pos-char}
  If $L$ is an $m$-periodic link and $\gcd(r,m)=1$, then
  \[\jones(L) = \sum_{d \mid m} f_{\frac{m}{d},r}\EJ_{d,r}(L),\]
  where $f_{\frac{m}{d},r}$ is defined in \eqref{eq:multiplicative_order}.
\end{proposition}
\begin{proof}
  The proof is almost identical to the proof of \cite[Theorem 3.2]{Politarczyk-Jones} once we use Lemma~\ref{obs:dim_irreducible_fin_char}.
\end{proof}

\smallskip
Until the end of the section we will assume that $m=p^n$ is a power of a prime.
Define difference Jones polynomials
\[\DJones_{\chi}(L) = \EJ_{p\cdot\chi,r}(L)-\EJ_{\chi,r}(L),\]
where $\chi \in C(p^{n},r)_{p^{s}}$, for $s < n$, and
\[p\cdot\chi = \{p \cdot a \colon a \in \chi\} \in C(p^{n},r)_{p^{s+1}}.\]
Suppose that $f_{p^n,r}=p^{n-1}(p-1)$, i.e. it is maximal. By Proposition~\ref{prop:maximal_multiplicative_order} we
infer that $C(p^n,r)_{p^s}$ consists of a single element for any $s>0$, hence
\[\EJ_{\chi}(L) = \EJ_{p^{n-s},r}(L), \quad \DJones_{\chi}(L) = \EJ_{p^{n-s-1},r}(L) - \EJ_{p^{n-s},r}(L).\]
In particular, we can rewrite Proposition~\ref{prop:fund-property-equiv-jones-pol-pos-char} as
 \[\jones(L) = \sum_{s=0}^{n-1}p^{s}(\EJ_{p^{s},r}(L) - \EJ_{p^{s+1},r}(L)) + p^{n}\EJ_{p^{n},r}(L).\]

Difference Jones polynomials in finite characteristic also satisfy an appropriate version of the skein relation.

\begin{theorem}\label{thm:skein-relation-finite-char}
Let $L$ be a $p^{n}$-periodic link. Suppose that $r$ has the maximal order in the multiplicative group mod $p^n$.
\begin{enumerate}
\item For any $\chi \in C(p^{n},r)_{p^{n-1}}$ the difference polynomial $\DJones_{\chi}(L)$ satisfies the following relation
\begin{align*}
 &q^{-2p^n} \DJones_{\chi}\left(\posOrbit\right) - q^{2p^n} 
 \DJones_{\chi}\left(\negOrbit\right) = \\
 &= \left(q^{-p^n} - q^{p^n}\right) \DJones_{\chi}\left(\orientResOrbit\right),
\end{align*}
where $\posOrbit$, $\negOrbit$ and $\orientResOrbit$ denote an orbit of positive, negative and 
orientation preserving resolutions of crossing, respectively.
\item If $\chi \in C(p^{n},r)_{p^{s}}$, then
\begin{align*}
 &q^{-2p^{n}}\DJones_{\chi}\left(\posOrbit\right) - 
 q^{2p^{n}}\DJones_{\chi}\left(\negOrbit\right) \equiv \\
 &\equiv \left(q^{-p^n} - q^{p^n}\right) \DJones_{\chi}\left(\orientResOrbit\right) 
 \pmod{q^{p^{s+1}} - q^{-p^{s+1}}}.
\end{align*}
\end{enumerate}
\end{theorem}
\begin{proof}
  The proof is a repetition of the proof of~\cite[Theorem 3.6]{Politarczyk-Jones} with essentially one difference, namely
the restriction functor for modules over $\F_r[\Z_{p^n}]$ is different than over $\Q[\Z_{p^n}]$ 
and this restriction functor is used in the proof of~\cite[Theorem 2.19]{Politarczyk-Jones} (see also~\cite[Theorem 5.11]{Politarczyk-Khovanov}),
which is a key result in the proof of~\cite[Theorem 3.6]{Politarczyk-Jones}.
The necessary modification of the proof in our case replaces the computation of the restriction functor over the base ring $\Z$ (or $\Q$) in~\cite[Proposition 2.7]{Politarczyk-Khovanov} by Lemma~\ref{lemma:restrictions}. We leave the details to the reader.
\end{proof}

\section{Proof of Theorem \ref{thm:periodicity_criterion}}\label{sec:bigproof}

We will focus on the case $\F\neq\F_2$. Suppose $p$ is invertible in $\F$. By Corollary~\ref{cor:trivial_on_scalars}
we have that $\EKh(L;\Lambda_{p^n})=\Kh(L;\F)$ hence:
\begin{equation}\label{eq:proof_1}
\KhP(L;\F)=\EKhP(L;\Lambda_{p^n}).
\end{equation}
By Proposition~\ref{prop:cartan} $\EKh(L;\Lambda_{p^n})=\Hom_{\Lambda_{p^n}}(\Lambda_{p^n},\Kh(L;\F))$. Write $\Lambda_{p^n}$ as a sum of irreducible summands
as in Proposition~\ref{prop:rational-representations} or Proposition~\ref{prop:representations-finite-characteristic} depending on the field $\F$. We focus on the case $\F=\F_r$ in the following, the case $\F=\Q$ is analogous. Remember that we assumed that the order of $r$ is maximal in the multiplicative group mod $p^n$, therefore by Corollary~\ref{cor:maximal_multiplicative_order}, $\F(\xi_{p^s})$ is irreducible for any $0 \leq s \leq p^{n}$, so
\begin{equation}\label{eq:homlambda}\Hom_{\Lambda_{p^n}}(\Lambda_{p^n},\Kh(L;\F))=\bigoplus_{s = 0 }^{n}\Hom_{\Lambda_{p^n}}(\F(\xi_{p^s}),\Kh(L;\F)).\end{equation}
Proposition~\ref{prop:cartan} identifies the summands of the right hand side of \eqref{eq:homlambda}
as $\EKh(L;\F(\xi_{p^s}))$. Then \eqref{eq:proof_1} combined with \eqref{eq:homlambda} induce the following equality for
the Khovanov polynomials. 
\begin{equation}\label{eq:proof_2}
\KhP(L;\F)=\sum_{s = 0}^n \EKhP(L;\F(\xi_{p^{s}})).
\end{equation}
The polynomials $\mathcal{P}_j$ of Theorem~\ref{thm:periodicity_criterion} are defined as $\mathcal{P}_s=\frac{1}{p^s-p^{s-1}}\EKhP(L;\F(\xi_{p^s}))$.
By Proposition~\ref{prop:divisib} $\mathcal{P}_s$ has integer coefficients (remember that $\varphi(p^s) = p^s - p^{s-1}$, for $s > 0$).

Consider now the Lee spectral sequence, whose first page is $\EKh(L;\F(\xi_{p^s}))$ and which converges to $\ELee(L;\F(\xi_{p^s}))$. 
By Proposition~\ref{prop:gradedpolynomial} we conclude that
\begin{equation}\label{eq:proof_3}
\EKhP(L;\F(\xi_{p^s}))=\ELeeP(L;\F(\xi_{p^s}))+(1+tq^4)\wt{\mathcal{S}}_{s1}+(1+tq^8)\wt{\mathcal{S}}_{s2}+\ldots
\end{equation}
for some polynomials $\mathcal{S}_{sj}$ with non-negative coefficients. The argument as in Proposition~\ref{prop:divisib} implies that the coefficients of 
$\wt{\mathcal{S}}_{sj}$ are divisible by $p^s-p^{s-1}$. Therefore $\mathcal{S}_{sj}=\frac{1}{p^s-p^{s-1}}\wt{\mathcal{S}}_{sj}$ has non-negative, integer
coefficients.

Corollary~\ref{cor:equiv_lee_homology_knot} computes the equivariant Lee homology of a knot. At the level of polynomials we obtain.
\begin{equation}\label{eq:proof_4}
\ELeeP(L;\F(\xi_{p^s}))=\begin{cases} q^{s(K;\F)}(q+q^{-1}),& s = 0\\ 0, & \textrm{otherwise}.\end{cases}
\end{equation}
Combining \eqref{eq:proof_4}, \eqref{eq:proof_3} and \eqref{eq:proof_2} we recover points (1), (2) and (4) of Theorem~\ref{thm:periodicity_criterion}. 

It remains to prove (3). By the definition of $\DJones_s(K)$:
\[\mathcal{P}_{s}|_{t = -1} - \mathcal{P}_{s+1}|_{t=-1} = \DJones_{s}(K).\]
Application of Theorem~\ref{thm:skein relation} or Theorem~\ref{thm:skein-relation-finite-char}, depending on the characteristic of $\F$, implies that polynomials $\mathcal{P}_{s}$, for $0 \leq s \leq n$, satisfy the desired congruences. This concludes the proof.

\section{Theorem~\ref{thm:periodicity_criterion} and other periodicity criteria}\label{sec:compare}

\subsection{Alexander polynomials and twisted Alexander polynomials of periodic knots}
One of the strongest criterion obstructing periodicity is due to Murasugi \cite{Murasugi}. We state it as follows.
\begin{theorem}[Murasugi criterion]\label{thm:ALEcrit}
Suppose $K\subset S^3$ is a $p$-periodic knot with $p$ a prime. Let $\Delta$ be the Alexander polynomial of $K$ and $\Delta_0$ be the
Alexander polynomial of the quotient knot $K/\Z_p$. Let $l$ be the absolute value of linking number of $K$ with the symmetry axis.
Then $\Delta_0|\Delta$ and up to multiplication by a power of $t$ we have
\begin{equation}\label{eq:congruence}
\Delta\equiv \Delta_0^p(1+t+\ldots+t^{l-1})^{p-1}\bmod p.
\end{equation}
\end{theorem}
A practical application of the Murasugi criterion is to find all possible factors of $\Delta$ over $\Z[t^{\pm}]$ and check if \eqref{eq:congruence} might hold
for some $l$.
 
There is the following generalization of the Murasugi criterion, due to Hillman, Livingston and Naik; see \cite{HillmanLivingstonNaik}.
\begin{theorem}[Twisted Alexander polynomial criterion]\label{thm:TAPcrit}
Suppose $K$ is $p$-periodic and $\ol{K}$ is the quotient knot. Assume that a knot group representation
$\ol{\rho}\colon\pi_1(S^3\setminus\ol{K})\to GL(n;\Z_p)$ (for some $n>0$) lifts to a representation
$\rho\colon\pi_1(S^3\setminus K)\to GL(n;\Z_p)$. Let $\Delta_{\rho}$ and $\Delta_{\ol\rho}$ be the twisted Alexander polynomials corresponding
to the two representations. Then there exists a Laurent polynomial $\delta$ with coefficients in $\Z_p$ such that up to a multiplication
by $t^k$ we have 
\[\Delta_\rho=\Delta_{\ol{\rho}}^p\delta^{p-1}\bmod p.\]
\end{theorem}
As it is explained in \cite{HillmanLivingstonNaik} the polynomial $\delta$ can be often computed in explicit examples. The practical use of Theorem~\ref{thm:TAPcrit}
depends on possibility of finding representations of the knot group of a potential quotient.  This is not a completely straightforward task. 
In \cite[Section 8]{HillmanLivingstonNaik} there is given an example, where the representation of the knot group of $K$ into a dihedral group is used.
According to \cite{ChuckLifting, Hartley}, representations of the knot group in the dihedral group $D_{2q}$ 
correspond to elements in $H_1(\Sigma(K),\Z_q)$, where $\Sigma(K)$ is the double branched cover. However if $K$ has Alexander polynomial $1$,
then $K$ does not admit any non-trivial representation in the dihedral group. Therefore one has to use more sophisticated representations in that case.
We did not apply the twisted Alexander polynomial criterion for our knots.

\subsection{Double branched covers of periodic knots}\label{sec:double}
In \cite{Naik}  Naik gave a criterion for periodicity of knots.
The criterion can be stated as follows; see \cite[Theorem 4]{Naik} and \cite[Section 2.1.3]{JabukaNaik}.
\begin{theorem}[Naik homological criterion]\label{thm:naik}
Suppose $K$ is periodic with period $p$ and let $k>1$. Suppose the $H_1(\Sigma^k(K);\Z)$ has $q$-torsion for some prime different than $p$ and let $l_q$
to be the least positive integer such that $q^{l_q}\equiv\pm 1\bmod p$. 
Then there exists non-negative integers $a_1,a_2,\ldots$ such that
\begin{equation}\label{eq:naik2}
H_1(\Sigma^k(K);\Z)_q/H_1(\Sigma^k(\ol{K}))_q=\Z_q^{2a_1l_q}\oplus\Z_{q^2}^{2a_2l_q}\oplus\ldots.
\end{equation}
\end{theorem}
The criterion can be implemented at two levels. Firstly if a knot passes Murasugi criterion, we obtain a potential Alexander polynomial $\Delta_0$ of the quotient.
Then for given $k$ we can find all prime divisors of $S_k=\frac{\prod\Delta(\zeta_i)}{\prod\Delta_0(\zeta_i)}$, where $\zeta_i$ are roots of unity
of order $k$ different than $1$. If $q|S_k$ we look for the maximal power $s_q$ such that $q^{s_q}|S_k$. Theorem~\ref{thm:naik} implies that $2l_q|s_q$.

If $q \nmid\prod\Delta_0(\zeta_i)$, for example if $\Delta_0\equiv 1$, then $H_1(\Sigma^k(\ol{K});\Z)_q$ is zero and then the quotient in 
\eqref{eq:naik2} depends only on the homology of $\Sigma^k(K)$, hence it can be determined from the Seifert matrix of $K$. This can strengthen the obstruction
coming from $2l_q|s_q$. For example, if the period $p=5$ and the determinant of $K$ is $121$, then 
we have two possibilities. Either $H_1(\Sigma^2(K);\Z)=\Z_{121}$ or $H_1(\Sigma^2(K);\Z)=\Z_{11}\oplus\Z_{11}$.
Naik's homology criterion implies that only the second case is allowed if $K$ is $5$--periodic.
Knots for which this phenomenon happens are listed in Table~\ref{tab:knots-passing-Naiks-crit}.
\begin{table}[h]
  \centering
  \begin{tabular}{llll}
\hline
    $12n504$    & $13n1487$  & $13n3926$   & $14n123$    \\
    $14n375$    & $14n7478$  & $14n9949$   & $15n35751$  \\
    $15n36061$  & $15n47753$ & $15n58627$  & $15n65643$  \\
    $15n97531$  & $15n36061$ & $15n139965$ & $15n142117$ \\
    $15n150771$ &            &             &\\\hline
  \end{tabular}
  \smallskip
  \caption{Prime non-alternating knots with 12--15 crossings that pass the Alexander polynomial (with $\Delta_0=1$)
and HOMFLYPT criterion, but $H_1(\Sigma^2(K);\Z)=\Z_{121}$,
so they fail the Naik's homology criterion for period $5$.}
  \label{tab:knots-passing-Naiks-crit}
\end{table}
%These knots pass the Przytycki--Traczyk HOMFLYPT criterion (see Theorem~\ref{thm:hom} below). The Murasugi criterion (Theorem~\ref{thm:ALEcrit})
%is satisfied if one sets $\Delta_0=1$. 
%Moreover the $11^2$ divides the 
%determinant of each of these knots. But $H_1$ of the double branched cover has a summand $\Z_{121}$ instead of $\Z_{11}\oplus\Z_{11}$,
%hence these knots are not periodic.

In \cite{JabukaNaik} Jabuka and Naik gave a criterion using $d$-invariants of Ozsv\'ath and Szab\'o. 
In short, given a knot $K\subset S^3$ one looks
at the $m$-fold branched cover $\Sigma_m(K)$, where $m$ is a power of prime. An action of $\F_{p}$ on $S^3$ fixing $K$ 
gives rise to an action of $\F_{p}$ on $\Sigma_m(K)$. This action induces a symmetry of $d$-invariants of Ozsv\'ath--Szab\'o. 
Therefore, if one can compute $d$-invariants
of the $m$-fold branched cover of the knot $K$, 
one can check if they satisfy the symmetry property, if not, $K$ is not periodic; we refer to \cite[Theorem 1.6 and Theorem 1.8]{JabukaNaik}
for the precise statement and the symmetry property.

The Jabuka--Naik criterion is particularly effective if $K$ is an alternating knot and $m=2$. In fact, using an explicit algorithm described in \cite{OS}
one can calculate $d$-invariants of a double branched cover of an alternating knot from its Goeritz matrix. Therefore the Jabuka--Naik criterion can
be effectively implemented on a computer. In \cite[Theorem 1.14]{JabukaNaik} there is given a list of all alternating knots with 12--15 crossings that
have prime period strictly greater than $3$. The efficacy of the Jabuka--Naik criterion is the main reason we focus on non-alternating knots in the present
article.

\subsection{HOMFLYPT polynomial of periodic links}
The last criterion we discuss here is the Przytycki--Traczyk HOMFLYPT criterion, see \cite{Przytycki-periodic,Traczyk}.
\begin{theorem}[HOMFLYPT criterion]\label{thm:hom}
Let $\mathcal{R}$ be a unital subring in $\Z[a^{\pm 1},z^{\pm 1}]$ generated by $a,a^{-1},\frac{a+a^{-1}}{z}$ and $z$. For a prime number $p$ let $\mathcal{I}_p$
be the ideal in $\mathcal{R}$ generated by $p$ and $z^p$. If a knot $K$ is $p$-periodic and $P(a,z)$ is its HOMFLYPT polynomial, then
\begin{equation}\label{eq:hom}
P(a,z)\equiv P(a^{-1},z)\bmod \mathcal{I}_p.
\end{equation}
\end{theorem} 
It is known that the HOMFLYPT polynomial of any link belongs to $\mathcal{R}$; see \cite[Lemma 1.1]{Przytycki-periodic}. It is therefore important
to notice that \eqref{eq:hom} is a congruence in $\mathcal{R}$ and not just in $\Z[a^{\pm 1},z^{\pm 1}]$. A method to check the HOMFLYPT criterion in practice
is given in \cite[Lemma 1.5]{Przytycki-periodic}.

\subsection{Prime non-alternating knots with $12$ to $15$ crossings}\label{sec:nonalternating}
The periodicity of composite links was basically settled by Sakuma \cite{Sakuma2}. Therefore in checking our criterion we focused
on prime knots.
We have applied the Alexander criterion and the HOMFLYPT criterion for prime non-alternating knots with $12$ to $15$ crossings, looking for knots with
period $5$. There were 262 knots that passed both criteria. We applied then the Naik criterion for the double branched covers. 
It obstructed further 133 knots. Three of the remaining knots fail the criterion from \cite{Politarczyk-Jones}. 
For the remainder 126 knot we
applied the Khovanov periodicity criterion Theorem~\ref{thm:periodicity_criterion}. We used the field $\F_2$ and then, for those that
pass the criterion over $\F_2$, checked also over the field $\F_3$.

Among these knots, 14 cases are particularly interesting, because they have
the Alexander polynomial equal to $1$. Therefore it is usually hard to find a non-trivial representation of the knot group and
apply the twisted Alexander polynomial obstruction (which we did not check). These knots are listed in Table~\ref{tab:candidates-with-trivial-alex}.

\begin{table}[h]
  \centering
  \begin{tabular}{llll}
	\hline
    $13n4582  \;  (\F_2)$ & $13n4591  \;  (\F_2)$ & $14n7708  \;  (\F_2)$ & $14n9290  \;  (\F_2)$ \\
    $15n49735 \;  (\F_3)$ & $15n50147 \;  (\F_3)$ & $15n62093 \;  (\F_3)$ & $15n62150 \;  (\F_3)$ \\
    $15n73226 \;  (\F_3)$ & $15n95983 \;  (\F_3)$ & $15n110439\;  (\F_3)$ & $15n135221\;  (\F_3)$ \\
    $15n135706\;  (\F_2)$ & $15n143825\;  (\F_2)$.
	\\\hline
  \end{tabular}
  \smallskip
  \caption{List of prime nonalternating knots with 12--15 crossings with trivial Alexander polynomial, passing HOMFLYPT criterion for $p = 5$ and
failing to Theorem~\ref{thm:periodicity_criterion}.
    The field in brackets indicates whether $\F_2$ or $\F_3$ version of our criterion obstructs $5$-periodicity. If a knot
is indicated as $\F_3$ it means that it passes the $\F_2$ criterion but fails to the $\F_3$ criterion.}
  \label{tab:candidates-with-trivial-alex}
\end{table}

The next $32$ knots, which are listed in Table~\ref{tab:candidates-nontrivial-alex}, have non-trivial Alexander polynomial, but pass the Murasugi criterion, the HOMFLYPT criterion and the Naik's homological criterion
for the double branched cover. It is possible, though, that a representation of the knot group in the dihedral group yields an obstruction
for the twisted Alexander polynomial. 

\begin{table}[h]
  \centering
  \begin{tabular}{llll}\hline
    $14n6530\;  (\F_3)$ & $14n6531\;  (\F_2)$ & $14n10342\; (\F_2)$ & $14n11504\; (\F_2)$ \\
    $14n22496\; (\F_3)$ & $14n23547\; (\F_3)$ & $14n23927\; (\F_3)$ & $14n23928\; (\F_2)$ \\
    $15n11273\; (\F_3)$ & $15n11529\; (\F_3)$ & $15n12330\; (\F_3)$ & $15n13166\; (\F_3)$ \\
    $15n15248\; (\F_3)$ & $15n19358\; (\F_3)$ & $15n35756\; (\F_3)$ & $15n39146\; (\F_3)$ \\
    $15n44135\; (\F_2)$ & $15n50407\; (\F_3)$ & $15n54555\; (\F_2)$ & $15n62879\; (\F_3)$ \\
    $15n69445\; (\F_3)$ & $15n74922\; (\F_3)$ & $15n81883\; (\F_3)$ & $15n81885\; (\F_3)$ \\
    $15n91712\; (\F_3)$ & $15n93084\; (\F_3)$ & $15n94281\; (\F_3)$ & $15n106946\; (\F_3)$ \\
    $15n108420\; (\F_3)$ & $15n110890\; (\F_3)$ & $15n131881\; (\F_3)$ & $15n137047\; (\F_3 )$\\\hline
  \end{tabular}
  \smallskip
  \caption{List of knots with nontrivial Alexander polynomial that pass the Alexander and HOMLYPT polynomial and Naik's periodicity criteria for $p=5$,
but fail to Theorem~\ref{thm:periodicity_criterion}.
    The field in brackets indicates whether $\F_2$ or $\F_3$ version of our criterion obstructs $5$-periodicity.}
  \label{tab:candidates-nontrivial-alex}
\end{table}

\subsection{Theorem~\ref{thm:periodicity_criterion} versus SnapPy}\label{sec:snappy}

In \cite{Weeks} Weeks implemented an algorithm for finding a canonical triangulation for a hyperbolic manifold and thus, to detect, whether two such manifolds
are diffeomorphic. The algorithm was implemented on a computer, as a SnapPea package, which has evolved to a package named SnapPy \cite{SnapPy}. For a
hyperbolic knot $K$ in $S^3$, SnapPy can compute the symmetry group of the complement of $K$. 
If the symmetry group of $S^3\setminus K$ does not have an element of order $m$, it follows that $K$ itself cannot have period $m$. This criterion is very effective, there are only $12$ prime non-hyperbolic non-alternating knots with $12$ to $15$ crossings.
These knots are listed in Table~\ref{tab:non-hyperbolic-knots}.

\begin{table}[h]
  \centering
  \begin{tabular}{llll}\hline
    $13n4587$   & $13n4639 $  & $14n21881$  & $14n22180$ \\
    $14n26039$  & $15n40211$  & $15n41185$  & $15n59184$ \\
    $15n115646$ & $15n124802$ & $15n142188$ & $15n156076$\\\hline
  \end{tabular}
  \caption{Prime non-hyperbolic nonalternating knots with $12$ to $15$ crossings.}
  \label{tab:non-hyperbolic-knots}
\end{table}
Among all the remaining (hyperbolic) knots only $8$ knots are such that
their complement admits an automorphism of order $5$. These are listed in Table~\ref{tab:hyperbolic-knots}.
\begin{table}[h]
  \centering
  \begin{tabular}{llll}\hline
    $12n887$   & $13n2833$   & $14n13191$  & $14n17159$  \\
    $15n99226$ & $15n112310$ & $15n142117$ & $15n166130$\\\hline
  \end{tabular}
  \caption{Prime hyperbolic knots with $12$-$15$ crossings whose symmetry groups contain an element of order $5$.}
  \label{tab:hyperbolic-knots}
\end{table}
In other words, with help of SnapPy one can obstruct periodicity of order $5$ for all but 20 prime knots with 12 to 15 crossings.

Applying the Murasugi's criterion leaves all knots but $15n41185$, $15n142117$, $15n166130$. All the three knots pass both the HOMFLYPT criterion
and Theorem~\ref{thm:periodicity_criterion}. The knot $15n142117$ fails the Naik's
homological criterion. The knot $15n41185$ is the torus knot $T(4,5)$, so it is periodic. The knot $15n166130$ is the only knot which remains unknown.

All this means that SnapPy is much more effective than the criterion for periodicity given by Theorem~\ref{thm:periodicity_criterion}. 
It is also much faster. Moreover sometimes our criterion works, when SnapPy cannot find
the symmetry group.

\begin{example}
Among the 20 knots listed in Tables~\ref{tab:non-hyperbolic-knots} and~\ref{tab:hyperbolic-knots}, Theorem~\ref{thm:periodicity_criterion}
rules out $12n887$, $15n112130$ and all the non-hyperbolic cases but $T(4,5)$.
\end{example} 

\subsection{Theorem~\ref{thm:periodicity_criterion} for periods 2 and 3}\label{sec:period3}

We begin with the following fact indicating that Theorem~\ref{thm:periodicity_criterion} works only for periods strictly greater than $3$.
\begin{proposition}[see \expandafter{\cite[Proposition 12.5]{Jones}}]
  The Jones polynomial of a knot $K$ satisfies
  $$(s^3-1)|(\jones(K)(s)-1).$$
\end{proposition}
Notice that Jones \cite{Jones} uses different normalization, in context of Khovanov homology one sets $s=q^2$, hence we obtain the following fact.
\begin{lemma}
  The Jones polynomial of a knot $\jones(K)$ satisfies
  $$\jones(K)(q)-\jones(K)(q^{-1})\equiv 0\bmod q^3-q^{-3}.$$
\end{lemma}
From this it follows that for any Khovanov polynomial $\KhP$ the conditions (1)--(4) of Theorem~\ref{thm:periodicity_criterion}
are satisfied if we set $\mathcal{P}_0=\KhP$ and $\mathcal{P}_k=0$, regardless on whether $K$ is actually $3$--periodic or not. In particular
\begin{corollary}
Theorem~\ref{thm:periodicity_criterion} cannot be used to obstruct a knot from being $3$--periodic.
\end{corollary}

\smallskip
As for period $2$ consider the polynomial $R(s)=\jones(K)(s)$ with the standard normalization. By \cite{Murakami} $R(\sqrt{-1})=(-1)^{\operatorname{Arf}(K)}$. 
It follows
that $R(s)-R(s^{-1})$ vanishes for $t=i$. It also trivially vanishes for $s=1,-1$. Therefore $R(s)-R(s^{-1})$ is divisible by $s^2-1$. Again, with the choice
$\mathcal{P}_0=\KhP$ and $\mathcal{P}_k=0$, every knot passes the criterion from Theorem~\ref{thm:periodicity_criterion}.

\smallskip
Theorem~\ref{thm:periodicity_criterion} can be used, though, to obstruct $3^n$-periodicity or $2^n$-periodicity for $n>1$.

\section{Periodic links}\label{sec:links}

It happens often in knot theory, that the case of links is more complicated than the case of knots. This is also the case for checking periodicity. To begin
with, if a link $L$ is periodic, the symmetry of $\R^3$ might permute some of the components of $L$, or fix all the components. An obstruction for
periodicity should take into account different possibilities of the action\dots

\subsection{Known criteria for periodic links}

In Section~\ref{sec:compare} we reviewed several criteria for periodic knots. Most of the criteria work for links, or can be generalized for links. However,
the periodicity criteria are usually less effective in case of links.

Before proceeding into discussion of periodicity criteria for links it is worth to notice that obstruction for periodicity can be obtained from linking numbers.

\begin{observation}
  Let $p$ be a prime and let $L = L_1 \cup L_2 \cup \ldots \cup L_k$ be a $p$-periodic link. If the action on components is trivial, then for any $1 \leq i < j \leq k$ we have $p \mid lk(L_i,L_j)$.

  If, on the other hand, none of the components is fixed, then if we denote the linking matrix of $L$ by $lk(L)$, then every nondiagonal entry of $lk(L)$ appears in the matrix with multiplicity divisible by $p$.
\end{observation}
The above observation can be easilly generalized to the mixed case.

The Murasugi--Przytycki--Traczyk criterion for the HOMFLYPT polynomials is already stated for links.

Snappy can detect periodicity of link complements. However there are much more non-hyperbolic links than prime links. For example, among 119150 links with 12--14
crossings, there are 998 non-hyperbolic links. This is still a small number, but it is much larger than for knots.

The following result generalizes the Murasugi criterion for the Alexander polynomial, Theorem~\ref{thm:ALEcrit}. 
For simplicity we restrict to the case when the group action fixes the components of the link.
\begin{theorem}[see \cite{Sakuma} and \expandafter{\cite[Theorem 1.10.1]{Turaev}}]\label{thm:ALEcritLinks}
Suppose $L=L_1\cup\ldots\cup L_n$ be an $n$-component $p$-periodic link with $p$ prime. 
Assume that the symmetry fixes the components.
%Assume that the symmetry group permutes cyclically the
%components $L_j,L_{j+k},\ldots,L_{j+(k-1)p}$ for $j=1,\ldots,k$ and fixes the components $L_{kp+1},\ldots,L_n$.
Let $L'$ be the quotient link.
%and denote it components by $L'_1,\ldots,L'_k,L'_{kp+1},\ldots,L_{n}$. 
%To simplify the formulae denote $\mathbf{t}=(t_1,\ldots,t_k,t_1,\ldots,t_k,\ldots,t_k,t_{kp+1},\ldots,t_n)$ and $\mathbf{t}_{red}=(t_1,\ldots,t_k,t_{kp+1},\ldots,t_n)$.
%The multivariable Alexander polynomials of $L'$ and $L$ satisfy the following relation.
\begin{equation}\label{eq:ALEcritLinks}
%\begin{split}
\Delta_{L'}(t_1,\ldots,t_n)|\Delta_{L}(t_1,\ldots,t_n).\\
%\Delta_{L}(t_1,\ldots,t_n)&\cong\Delta_{L'}(t_1,\ldots,t_n)^p\rho^{p-1}\bmod p,
%\end{split}
\end{equation}
%where $\rho$ is an explicit polynomial depending on the linking of the components of $L$ with the symmetry axis.
\end{theorem}

It is possible, in theory, to generalize Theorem~\ref{thm:ALEcritLinks} 
to obtain a criterion for algebraic links.
given in \cite{Turaev} can be generalized for twisted Alexander polynomials as well. Moreover,
if the group action permutes some components of $L$, an analogue of Theorem~\ref{thm:ALEcritLinks} can be given. 

In order to generalize Naik's criterion we use the following fact.
\begin{lemma}\label{thm:branchedlink}
Suppose $L$ is a $p$-periodic link with non-zero determinant. Let $L'$ be the quotient. Then $\det(L')|\det(L)$.
\end{lemma}
\begin{proof}
By \cite{Sakuma} we have that $\det(L)=2\Delta(-1,\ldots,-1)$, so the result follows from \eqref{eq:ALEcritLinks}.
\end{proof}

Lemma~\ref{thm:branchedlink} allows us to give the following result, which is known to the experts.
\begin{corollary}
The Naik's criterion (Theorem~\ref{thm:naik}) works also for links. The Jabuka--Naik method allows to obstruct periodicity of quasi-alternating links and not
only quasi-alternating knots.
\end{corollary}
\begin{remark}
For links the determinant can be zero, and then both the Naik's criterion and the Jabuka--Naik criterion fail to give an obstruction.
\end{remark}
The way to apply Naik's criterion for links is rather straightforward. First take the multivariable Alexander polynomial of $L$. Find all its divisors
over $\Z[t_1^{\pm 1},\ldots,t_n^{\pm 1}]$. As this is a multivariable polynomial, there should not be too many of them. Each divisor $\Delta'$
satisfying the symmetry property $\Delta'(t_1^{-1},\ldots,t_n^{-1})=\pm\Delta'(t_1,\ldots,t_n)$ 
can potentially be the Alexander polynomial of the quotient. For such $\Delta'$ compute the quotient $h_{free}=\Delta(-1,\ldots,-1)/\Delta'(-1,\ldots,-1)$.
This $h_{free}$ is the candidate for the par of $H_1(\Sigma(L);\Z)$ on which the symmetry group acts freely. Then we act as descibed in Section~\ref{sec:double}.

\subsection{Periodicity obstruction for links}\label{sec:period_links}
The proof of Theorem~\ref{thm:periodicity_criterion} uses the fact that $L$ is a knot (and not a link) only in one place. Namely, we
use the computation of $\ELeeP(L;\F(\xi_{p^s}))$ from Corollary~\ref{cor:equiv_lee_homology_knot}, therefore the theorem can be extended to the case of links with trivial action on components without any difficulties.

For a general link $L$, in order to determine individual groups $\ELee^{i}(L, M)$ (the same reasoning applies to $\EBN(L, M)$), for $M$ either $\Q(\xi_{m/d})$ or $V_{\chi}$, for some $\chi \in C(m, r)_{d}$, it is sufficient to look at canonical generators of $\Lee^{i}(L)$. Proposition~\ref{prop:equiv-lee-homology-computation-semisimple-case} states that $\ELee^{i}(L, M)$ is generated over $\F$ by certain linear combinations of generators from orbits with isotropy group $G$ such that $G \subset \Z_{d} \subset \Z_{m}$, see \cite{CurtisReiner}. In particular the equivariant Lee polynomial $\ELeeP(L;M)$ can be determined using the $s$--invariants of $L$,
the linking numbers of components of $L$ and the action of the symmetry group on the components of $L$.

\begin{example}\label{ex:link-periodicity}
  Consider Borromean rings. It is a $3$--periodic link. Choose an orientation preserved by the rotational $\Z_3$ symmetry. The symmetry interchanges components of $L$ and the
linking numbers between components are $0$. There are $8$ canonical generators which correspond to $8$ orientations. The chosen orientation $\mathcal{O}_0$ and its reverse $\overline{\mathcal{O}}_0$ are invariant. The remaining orientations can be divided into two orbits $\mathfrak{O}_1$ and $\mathfrak{O}_2$. Therefore, $\ELee(L, \Q)$ is spanned by $x_{\mathcal{O}_1}$, $x_{\mathcal{O}_2}$, $\frac{1}{3} \sum_{\mathcal{O} \in \mathfrak{O}_1} x_{\mathcal{O}}$ and $\frac{1}{3} \sum_{\mathcal{O} \in \mathfrak{O}_2} x_{\mathcal{O}}$. This implies that
  $$\ELeeP(L, \Q) = \ELeeP(L, \Q(\xi_3)) = 2q + 2q^{-1},$$
  because coefficients of $\ELeeP(L, \Q(\xi_3))$ are always even. %Therefore, with some help from the skein spectral sequence from~\cite{Politarczyk-Khovanov}, 
\end{example}

%As a corollary we get the following generalization of Theorem~\ref{thm:periodicity_criterion}.
The proof of Theorem~\ref{thm:periodicity_criterion} given in Section~\ref{sec:bigproof}
can be generalized for links. The following formulation requires the knowledge of the equivariant Lee polynomial $\ELeeP$, but this can be computed
using methods of Section~\ref{sec:lee}.

\begin{theorem}\label{thm:periodic_two_links}
Suppose $L$ is a $p^n$-periodic link with $p$ odd. Assume that $\F$ is as in Theorem~\ref{thm:periodicity_criterion}. The Khovanov polynomial
$\KhP(L;\F)$ decomposes as $\KhP(L;\F)=\mathcal{P}_0+\sum_{j=1}^n (p^j-p^{j-1})\mathcal{P}_j$, where
$$\mathcal{P}_{j} = \frac{1}{p^j-p^{j-1}} \ELeeP(L,\F(\xi_{p^j})) + \sum_{j=1}^{\infty}(1+tq^{2cj})\mathcal{S}_{kj}(t,q),$$
for $0 \leq j \leq n$ and $\mathcal{S}_{k,j}(t,q)$ have nonnegative coefficients. Moreover, polynomials $\mathcal{P}_{j}$ satisfy conditions~\ref{item5:main_thm_1} and~\ref{item6:main_thm_1} from Theorem~\ref{thm:periodicity_criterion}.
\end{theorem}
If the action of the symmetry group fixes the components, then the equivariant Lee polynomial is equal to the standard Lee polynomial for the trivial representation,
and it is equal to zero for all others. Therefore we obtain the following corollary.
\begin{corollary}\label{cor:fixed}
  If $L$ is as in the above theorem and the action of $\Z_{p^j}$ is trivial on components, then we have a decomposition of the Khovanov polynomial of $L$
  $$\KhP(L;\F)=\mathcal{P}_0+\sum_{j=1}^n (p^j-p^{j-1})\mathcal{P}_j$$
  where
  \begin{align*}
    \mathcal{P}_0 &= \LeeP(L,\F) + \sum_{j=1}^\infty (1+tq^{2cj})\mathcal{S}_{0j}(t,q), \\
    \mathcal{P}_{j} &= \sum_{k=1}^\infty (1+tq^{2ck})\mathcal{S}_{k,j}(t,q),
  \end{align*}
  satisfy conditions~\ref{item5:main_thm_1} and~\ref{item6:main_thm_1} from Theorem~\ref{thm:periodicity_criterion}.
\end{corollary}

\subsection{Examples for links}\label{sec:exlinks}
We have applied the following periodicity criteria for periods $5$ and $7$ looking for links admitting a group action fixing the link components.
\begin{itemize}
\item The HOMFLYPT critierion;
\item The Naik's criterion as described above;
\item The Snappy criterion.
\end{itemize}
Out of 120572 links with 12 to 15 crossings, 39 pass all the three above criteria for period $5$. 
Further 13 links fail to the linking number criteria.
Among the remaining links, the following 11 links
L11n436, 
L12n1868,  
L12n1933, 
L14n47791, 
L14n48169,
L14n48174,
L14n48450, 
L14n51728, 
L14n52589, 
L14n52622 and 
L14n54427 are obstructed by Corollary~\ref{cor:fixed}. All these links have determinant $0$, therefore the Naik's criterion for the double branched cover
does not give any obstruction, moreover the Jabuka--Naik criterion does not apply either.

As for period 7, there are 10 knots among those that 120572 knots that pass the three above criteria. Exactly one link, L14n64054, is obstructed by the Khovanov
polynomial criterion. This link also has determinant $0$.

\smallskip
We remark that 
it is possible to push the Sakuma' criterion for the Alexander polynomials (Theorem~\ref{thm:ALEcritLinks}) 
criterion to obtain a stronger obstruction from the Alexander polynomials than just divisibility in \eqref{eq:ALEcritLinks}, more in spirit of 
Theorem~\ref{thm:ALEcrit}. The criterion might be more complicated, because it might involve different linking numbers between different components and
the rotation axis. We did not check this enhanced criterion.

%We get a decomposition of $\KhP(L, \Q)$ of the following form
%  $$\KhP(L, \Q) = \mathcal{P}_{0} + 2 \mathcal{P}_1,$$
%  where
%  \begin{align*}
%    \mathcal{P}_0 &= 2(q + q^{-1}) + (1 + tq^4) (t^{-3}q^{-7} + t^2 q^3), \\
%    \mathcal{P}_1 &= (q + q^{-1}) + (1 + tq^4) (t^{-2} q^{-5} + tq).
%  \end{align*}

\appendix

\section{Review of representation theory}\label{appendix}

The purpose of this section is to gather results from representation theory needed in the remainder part of the paper. We will always denote by $\F$ either the field of rationals $\Q$ or the finite field $\F_{r}$, for a prime $r$.

Let $\Z_{m}$ be a finite cyclic group of order $m$. We will be interested in representations of this group over $\Q$ and over $\F_{r}$, for some prime $r$ such that $\gcd(r,m)=1$. Notice that $\gcd(r,m)=1$ 
if and only if, $m$ is invertible in $\F_{r}$.

We start with the following fundamental results.

\begin{proposition}[Maschke's Theorem]\label{prop:maschke-thm}
  The group algebra $\F[\Z_{m}]$ is semisimple if and only if, $m$ is invertible in $\F$.
\end{proposition}
As a corollary we obtain a result which is extensively used throughout the paper.
\begin{corollary}[see \expandafter{\cite[Theorem 4.2.2]{Weibel}}]\label{cor:extvanish}
If $m$ is invertible in $\F$ and $\Lambda_m=\F[\Z_m]$, then $\Ext^i_{\Lambda_m}(M;N)=0$ for any $i>0$ and any $\Lambda_m$-modules $M$ and $N$.
\end{corollary}

\begin{lemma}[Schur's Lemma]\label{prop:schur-lemma}
  Let $\F$ be as above. Suppose that $M$ and $N$ are irreducible $\Z_{m}$-modules. If $M \not\cong N$, then
  $$\Hom_{\F[\Z_{m}]}(M,N) = 0.$$
  On the other hand if $M \cong N$, then every non-zero element of $\Hom_{\F[\Z_{m}]}(M,N)$ is an isomorphism.
\end{lemma}

From now on we will assume that $m$ is invertible in $\F$. Since $\F[\Z_{m}] = \F[X]/(X^m-1)$ is semisimple, the decomposition of the group algebra into irreducible representations is governed by the decomposition of the polynomial $X^{m}-1$ into irreducible factors over $\F$. This polynomial decomposes over $\Z$ into product of \textit{cyclotomic polynomials}

\[\Phi_{d}(X) = \prod_{\stackrel{1 \leq k < d}{\gcd(k,d) = 1}}(X - \xi_{d}) \in \Z[X], \quad d \mid m, \]
where $\xi_{d}$ denotes a primitive root of unity of order $d$. In particular, we have a decomposition
\[\F[\Z_{m}] = \bigoplus_{d \mid m} \F[X]/(\Phi_{d}(X)),\]
the above summands need not be irreducible though. 

\begin{proposition}\label{prop:rational-representations}
  The group algebra $\Q[\Z_{m}]$ admits the following decomposition into a direct sum of irreducible summands
  \[\Q[\Z_{m}] = \bigoplus_{d \mid m} \Q(\xi_{d}),\]
  where
  \[\Q(\xi_{d}) = \Q[X]/(\Phi_{d}(X)), \quad \xi_{d} = \exp\left(\frac{2 \pi i}{d}\right).\]
\end{proposition}

Define the \emph{Euler's totient function} to be
\begin{equation}
  \label{eq:Eulers-totient-function}
  \varphi(d) = \#\{k \in \Z \mid 1 \leq k \leq d-1, \quad \gcd(d,k) = 1\}.
\end{equation}
It is not hard to check that $\dim_{\Q} \Q(\xi_d) = \varphi(d).$ Moreover, if $p$ is a prime and $k \geq 1$, then $\varphi(p^k) = p^k - p^{k-1}$.

In finite characteristic the situation is a bit different. Irreducible representations are parametrized by \textit{cyclotomic cosets} mod $r$ in $\Z_{m}$.  They are defined as equivalence classes of the following relation
  \[a \sim_{r} b \iff \exists_{k} \quad a\cdot r^{k} \equiv b \pmod{m}.\]
Denote by $C(m,r)$ the set of cyclotomic cosets mod $r$ in $\Z_{m}$. If $d \mid m$ we can distinguish a subset $C(m,r)_{d} \subset C(m,r)$ consisting of cyclotomic cosets $\chi$ such that any $a \in \chi$ satisfies $\gcd(a,m) = d$.

\begin{definition}
  To every cyclotomic coset $\chi \in C(m,r)$ we can assign a polynomial in $\F_{r}[X]$
  \[f_{\chi}(X) = \prod_{u \in \chi} (X - \xi_{m}^{u}),\]
  where $\xi_{m}$ denotes a primitive root of unity in some finite extension $\K$ of $\F_{r}$.
\end{definition}

\begin{proposition}\label{prop:cyclotomic-pol-in-finite-fields}
  If $d \mid m$ then the cyclotomic polynomial $\Phi_{\frac{m}{d}}(X)$ decomposes in $\F_{r}[X]$ in the following way
  \[\Phi_{\frac{m}{d}}(X) = \prod_{\chi \in C(m,r)_{d}} f_{\chi}(X).\]
\end{proposition}
\begin{proof}
  Let $\F_{r} \subset \K$ be an extension that contains all roots of unity of order $m$. The Galois group $Gal(\K/\F_{r})$ is generated by the Frobenius map
  \[F \colon \K \to \K, \quad F(a) = a^{r}.\]
  The cyclotomic polynomial $\Phi_{m}(X)$ splits into linear factors over $\K$
  \[X^{m} - 1 = \prod_{0 \leq k \leq m-1} (X - \xi_{m}^{k}),\]
  therefore irreducible factors of $(X^{m}-1)$ over $\F_{r}$ correspond to the orbits of the action of $Gal(\K/\F_{r})$ on the set of roots of unity of order $m$. It is easy to verify that every such orbit is of the form
  \[\{\xi_{m}^{k} \colon k \in \chi\},\]
  for some fixed $\chi \in C(m,r)$.
\end{proof}

\begin{definition}\label{def:cyclotomic-cosets-modules}
  Let $\chi \in C(m,r)$, then define the $\F_{r}[\Z_{m}]$-module
  \[V_{\chi} = \F_{r}[X]/(f_{\chi}(X)).\]
  If $t \in \Z_{m}$ is a fixed generator, then $t$ acts on $V_{\chi}$ by multiplication by $X$.
\end{definition}
Let, $f_{d,r}$ denote the multiplicative order of $r$ mod $d$, i.e. $f_{d,r}$ is the smallest positive integer such that
\begin{equation}
  \label{eq:multiplicative_order}
  r^{f_{d,r}} \equiv 1 \pmod{d}.
\end{equation}
\begin{lemma}\label{obs:dim_irreducible_fin_char}
  We have $\dim_{\F_{r}} V_{\chi} = f_{\frac{m}{d},r}$, provided that $\chi \in C(m,r)_{d}$.
\end{lemma}
\begin{proof}
  Follows directly from the definition.
\end{proof}

\begin{proposition}\label{prop:representations-finite-characteristic}
  The group algebra $\F_{r}[\Z_{m}]$ admits the following decomposition into irreducible modules
  \[\F_{r}[\Z_{m}] = \bigoplus_{\chi \in C(m,r)} V_{\chi}.\]
  Moreover, if we define a $\Z_{m}$-module $\F_r[\xi_{d}] = \F_r[X]/(\Phi_{d}(X))$, for $d \mid m$, then
  \[\F_r[\xi_{\frac{m}{d}}] = \bigoplus_{\chi \in C(m,r)_{d}} V_{\chi}.\]
\end{proposition}

Let us now restrict our attention to the case when $m = p^{n}$ is a power of a prime such that $r \neq p$. For any $\chi \in C(p^{n},r)$ define
\[p^{s} \cdot \chi = \{p^{s} \cdot a \colon a \in \chi\}.\]
It is easy to verify that if $\chi \in C(p^{n},r)_{p^{t}}$, then $p^{s}\cdot \chi \in C(p^{n},r)_{p^{s+t}}$. Moreover, observe that the following map is a bijection
\begin{align}
  C(p^{n},r)_{p^{t}} &\to C\left(p^{n-t},r\right)_{1}, \nonumber \\
  \chi &\mapsto \left\{\frac{a}{p^{t}} \colon a \in \chi\right\}. \label{eq:cycl_cosets_bijection} 
\end{align}

The multiplicative subgroup of invertible elements in $\Z_{p^{n}}$ is cyclic of order $(p-1)p^{n-1}$, therefore $f_{p^n,r} \leq (p-1)p^{n-1}$.
\begin{proposition}\label{prop:maximal_multiplicative_order}
  If $f_{p^n,r} = (p-1)p^{n-1}$, then for any $0 \leq s \leq n-1$ the set $C(p^{n},r)_{p^{s}}$ consists of a single element.
\end{proposition}
\begin{proof}
By \eqref{eq:cycl_cosets_bijection} it is enough to prove the result only for $C(p^{n},r)_{1}$.

  Notice that if the assumption of proposition is satisfied, then $r$ generates multiplicative subgroup of invertible elements of $\Z_{p^n}$, hence if $k$ is relatively prime to $p^n$ satisfies
  \[k \equiv r^l \pmod{p^n}\]
  for some $l$. This concludes the proof.
\end{proof}
\begin{corollary}\label{cor:maximal_multiplicative_order}
  If $f_{p^n,r} = (p-1)p^{n-1}$, then
  \[V_{\chi} = \F_{r}(\xi_{p^s})\]
  provided that $\chi \in C(p^n,r)_{p^{n-s}}$.
\end{corollary}

To conclude this section we will study the behavior of the restriction functor $\res{\Z_{p^{n}}}{\Z_{p^{s}}}$, for $s < n$, in finite characteristic. As a quick reminder (for the much more detailed account of restriction functors see for example \cite{CurtisReiner}) restriction functor assigns to any $\Lambda_{p^{n}}$-module $M$ a $\Lambda_{p^s}$-module $\res{\Z_{p^{n}}}{\Z_{p^{s}}} M$ by restricting the action of $\Lambda_{p^n}$ to the subalgebra $\Lambda_{p^s} \subset \Lambda_{p^n}$. The following proposition is used in the proof of Theorem~\ref{thm:skein-relation-finite-char}.

\begin{lemma}\label{lemma:restrictions}
  Let $\chi \in C(p^{n},r)_{p^t}$, then
  \[\res{\Z_{p^{n}}}{\Z_{p^{s}}} V_{\chi} =
    \begin{cases}
      \F_{r}^{\dim_{\F_{r}}V_{\chi}}, s \leq t, \\
      V_{p^{n-s} \cdot \chi}^{\alpha}, s > t,
    \end{cases}
  \]
  where $\alpha = \frac{\dim_{\F_{r}}V_{\chi}}{\dim_{\F_{r}}V_{p^{n-s}\cdot \chi}}$. Moreover, we implicitly treat $p^{n-s} \cdot \chi$ to be a cyclotomic coset in $C(p^{s},r)_{p^{t}}$ yielding a representation of $\Z_{p^{s}}$.
\end{lemma}
\begin{proof}
  Notice that $\Z_{p^{s}} \subset \Z_{p^n}$ is generated by $t^{p^{n-s}}$, where $t$ denotes a fixed generator of $\Z_{p^{n}}$. Therefore, in order to determine the restriction of $V_{\chi}$ to $\Z_{p^{s}}$ we need to analyze the action of $t^{p^{n-s}}$ on this representation. From Propositions \ref{prop:cyclotomic-pol-in-finite-fields} and \ref{prop:representations-finite-characteristic} it is easy to verify that the desired formula holds (compare it with calculations in \cite{Politarczyk-Jones}).
\end{proof}

\bibliographystyle{natplain}
\bibliography{biblio2}

% \bib, bibdiv, biblist are defined by the amsrefs package.
\begin{bibdiv}
\begin{biblist}

\bib{BarNatan}{article}{
      author={Bar-Natan, D.},
       title={Khovanov's homology for tangles and cobordisms},
        date={2005},
     journal={Geom. Topol.},
      volume={9},
       pages={1443\ndash 1499},
}

\bib{SnapPy}{misc}{
      author={Culler, M.},
      author={Dunfield, N..},
      author={Goerner, M.},
      author={Weeks, J.},
       title={Snap{P}y, a computer program for studying the geometry and
  topology of $3$-manifolds},
        note={available at \url{http://snappy.computop.org}},
}

\bib{CurtisReiner}{book}{
      author={Curtis, C.},
      author={Reiner, I.},
       title={Methods of representation theory. {V}ol. {I}},
      series={Wiley Classics Library},
   publisher={John Wiley \& Sons, Inc., New York},
        date={1990},
}

\bib{Hartley}{article}{
      author={Hartley, R.},
       title={Metabelian representations of knot groups},
        date={1979},
        ISSN={0030-8730},
     journal={Pacific J. Math.},
      volume={82},
      number={1},
       pages={93\ndash 104},
         url={http://projecteuclid.org/euclid.pjm/1102785063},
}

\bib{HillmanLivingstonNaik}{article}{
      author={Hillman, J.},
      author={Livingston, C.},
      author={Naik, S.},
       title={Twisted {A}lexander polynomials of periodic knots},
        date={2006},
        ISSN={1472-2747},
     journal={Algebr. Geom. Topol.},
      volume={6},
       pages={145\ndash 169 (electronic)},
}

\bib{JabukaNaik}{article}{
      author={Jabuka, S.},
      author={Naik, S.},
       title={Periodic knots and {H}eegaard {F}loer correction terms},
        date={2016},
        ISSN={1435-9855},
     journal={J. Eur. Math. Soc. (JEMS)},
      volume={18},
      number={8},
       pages={1651\ndash 1674},
}

\bib{Jones}{article}{
      author={Jones, V.},
       title={Hecke algebra representations of braid groups and link
  polynomials},
        date={1987},
     journal={Ann. of Math.},
      volume={126},
      number={2},
       pages={335\ndash 388},
}

\bib{Khovanov}{article}{
      author={Khovanov, M.},
       title={Link homology and {F}robenius extensions},
        date={2006},
     journal={Fund. Math.},
      volume={190},
       pages={179\ndash 190},
}

\bib{Lee}{article}{
      author={Lee, E.},
       title={An endomorphism of the {K}hovanov invariant},
        date={2005},
     journal={Adv. Math.},
      volume={197},
      number={2},
       pages={554\ndash 586},
}

\bib{ChuckLifting}{article}{
      author={Livingston, C.},
       title={Lifting representations of knot groups},
        date={1995},
        ISSN={0218-2165},
     journal={J. Knot Theory Ramifications},
      volume={4},
      number={2},
       pages={225\ndash 234},
         url={http://dx.doi.org/10.1142/S0218216595000120},
}

\bib{McCleary}{book}{
      author={McCleary, J.},
       title={A user's guide to spectral sequences},
      series={Cambridge Studies in Advanced Mathematics},
   publisher={Cambridge University Press, Cambridge},
        date={2001},
      volume={58},
}

\bib{Murakami}{article}{
      author={Murakami, H.},
       title={A recursive calculation of the {A}rf invariant of a link},
        date={1986},
        ISSN={0025-5645},
     journal={J. Math. Soc. Japan},
      volume={38},
      number={2},
       pages={335\ndash 338},
}

\bib{Murasugi}{article}{
      author={Murasugi, K.},
       title={On periodic knots},
        date={1971},
     journal={Comment. Math. Helv.},
      volume={46},
       pages={162\ndash 174},
}

\bib{Naik}{article}{
      author={Naik, S.},
       title={New invariants of periodic knots},
        date={1997},
        ISSN={0305-0041},
     journal={Math. Proc. Cambridge Philos. Soc.},
      volume={122},
      number={2},
       pages={281\ndash 290},
         url={http://dx.doi.org/10.1017/S0305004197001801},
}

\bib{OS}{article}{
      author={Ozsv\'ath, P.},
      author={Szab\'o, Z.},
       title={On the {H}eegaard {F}loer homology of branched double-covers},
        date={2005},
        ISSN={0001-8708},
     journal={Adv. Math.},
      volume={194},
      number={1},
       pages={1\ndash 33},
}

\bib{Politarczyk-Jones}{misc}{
      author={Politarczyk, W.},
       title={Equivariant {J}ones polynomials of periodic links},
        date={2015},
        note={arXiv:1504.03462, to appear in J. of Knot Theory Ramifications},
}

\bib{Politarczyk-Khovanov}{misc}{
      author={Politarczyk, W.},
       title={Equivariant {K}hovanov homology of periodic links},
        date={2015},
        note={arXiv:1504.03462},
}

\bib{Przytycki-periodic}{article}{
      author={Przytycki, J.},
       title={On {M}urasugi's and {T}raczyk's criteria for periodic links},
        date={1989},
     journal={Math. Ann.},
      volume={283},
      number={3},
       pages={465\ndash 478},
}

\bib{Rasmussen}{article}{
      author={Rasmussen, J.},
       title={Khovanov homology and the slice genus},
        date={2010},
     journal={Invent. Math.},
      volume={182},
      number={2},
       pages={419\ndash 447},
}

\bib{Sakuma}{article}{
      author={Sakuma, M.},
       title={On the polynomials of periodic links},
        date={1981},
        ISSN={0025-5831},
     journal={Math. Ann.},
      volume={257},
      number={4},
       pages={487\ndash 494},
}

\bib{Sakuma2}{article}{
      author={Sakuma, M.},
       title={Periods of composite links},
        date={1981},
        ISSN={0385-633X},
     journal={Math. Sem. Notes Kobe Univ.},
      volume={9},
      number={2},
       pages={445\ndash 452},
}

\bib{KnotKit-WP}{misc}{
      author={Seed, C.},
       title={Knotkit},
        date={2016},
        note={with a module checking periodicity written by W.~Politarczyk,
  available at \url{https://github.com/wpolitarczyk/knotkit}},
}

\bib{KnotKit}{misc}{
      author={Seeed, C.},
       title={Knotkit},
        date={2016},
        note={available at \url{https://github.com/cseed/knotkit}},
}

\bib{Traczyk}{article}{
      author={Traczyk, P.},
       title={{$10_{101}$} has no period {$7$}: a criterion for periodic
  links},
        date={1990},
        ISSN={0002-9939},
     journal={Proc. Amer. Math. Soc.},
      volume={108},
      number={3},
       pages={845\ndash 846},
}

\bib{Traczyk2}{article}{
      author={Traczyk, P.},
       title={Periodic knots and the skein polynomial},
        date={1991},
        ISSN={0020-9910},
     journal={Invent. Math.},
      volume={106},
      number={1},
       pages={73\ndash 84},
}

\bib{Turaev}{article}{
      author={Turaev, V.},
       title={Reidemeister torsion in knot theory},
        date={1986},
     journal={Uspekhi Mat. Nauk},
      volume={41},
      number={1(247)},
       pages={97\ndash 147, 240},
}

\bib{Turner}{article}{
      author={Turner, P.},
       title={Calculating {B}ar-{N}atan's characteristic two {K}hovanov
  homology},
        date={2006},
     journal={J. Knot Theory Ramifications},
      volume={15},
      number={10},
       pages={1335\ndash 1356},
}

\bib{Watson}{misc}{
      author={Watson, L.},
       title={Khovanov homology and the symmetry group of a knot},
        date={2013},
        note={arXiv:1311.1085 to appear in Adv. Math.},
}

\bib{Weeks}{article}{
      author={Weeks, J.},
       title={Convex hulls and isometries of cusped hyperbolic
  {$3$}-manifolds},
        date={1993},
        ISSN={0166-8641},
     journal={Topology Appl.},
      volume={52},
      number={2},
       pages={127\ndash 149},
}

\bib{Weibel}{book}{
      author={Weibel, C.},
       title={An introduction to homological algebra},
      series={Cambridge Studies in Advanced Mathematics},
   publisher={Cambridge University Press, Cambridge},
        date={1994},
      volume={38},
}

\bib{Zhang}{article}{
      author={Zhang, M.},
       title={A rank inequality for the annular khovanov homology of 2-periodic
  links},
        date={2017},
        note={arXiv:1707.03279},
}

\end{biblist}
\end{bibdiv}
\end{document}